\pdfoutput=1
\documentclass[leqno]{shinyart}

\usepackage[utf8]{inputenc}
\usepackage{enumitem}
\newlist{assumption}{enumerate}{1}
\setlist[assumption]{label=(\textsc{a}\arabic*)}
\crefname{assumptioni}{Assumption}{Assumptions}

\usepackage{shinybib}
\usepackage{booktabs}
\usepackage[exponent-product=\cdot]{siunitx}

\newcommand{\R}{\mathbb{R}}
\newcommand{\N}{\mathbb{N}}
\newcommand{\1}{\mathbb{1}}

\DeclareMathOperator{\sign}{\mathrm{sign}}

\undef\div
\DeclareMathOperator{\div}{\mathrm{div}}
\DeclareMathOperator{\Id}{\mathrm{Id}}
\def\weakto{\rightharpoonup}
\renewcommand{\epsilon}{\varepsilon}

\usepackage{graphicx}
\graphicspath{{./pics/}}

\title{Optimal control of a non-smooth quasilinear elliptic equation}
\author{Christian Clason\thanks{Faculty of Mathematics, University of Duisburg-Essen, Thea-Leymann-Strasse 9, 45127 Essen, Germany\newline%
    (\email{christian.clason@uni-due.de}, \email{huu.vu@uni-due.de}, \email{arnd.roesch@uni-due.de})}
    \and Vu Huu Nhu\footnotemark[1]
    \and Arnd Rösch\footnotemark[1]
}

\begin{document}
\maketitle

\begin{abstract}
    This work is concerned with an optimal control problem governed by a non-smooth quasilinear elliptic equation with a nonlinear coefficient in the principal part that is locally Lipschitz continuous and directionally but not Gâteaux differentiable. This leads to a control-to-state operator that is directionally but not Gâteaux differentiable as well. Based on a suitable regularization scheme, we derive C- and strong stationarity conditions. Under the additional assumption that the nonlinearity is a $PC^1$ function with countably many points of nondifferentiability, we show that both conditions are equivalent. Furthermore, under this assumption we derive a relaxed optimality system that is amenable to numerical solution using a semi-smooth Newton method. This is illustrated by numerical examples.

    \medskip

    \noindent\textcolor{structure}{Key words}\quad Optimal control, non-smooth optimization, optimality system, quasilinear elliptic equation.
\end{abstract}

\section{Introduction}

This work is concerned with the non-smooth quasilinear elliptic optimal control problem
\begin{equation*}
    \left\{
        \begin{aligned}
            \min_{u\in L^p(\Omega), y\in H^1_0(\Omega)} &J(y,u) \\
            \text{s.t.} \quad &-\div [a(y)\nabla y + b(\nabla y)] = u \quad \text{in } \Omega
        \end{aligned}
    \right.
\end{equation*}
with a Fréchet differentiable functional $J:H^1_0(\Omega)\times L^p(\Omega)\to\R$ for a domain $\Omega\subset \R^N$, $p > \frac{N}{2}$, a non-smooth function $a: \R \to \R$, and a continuously differentiable vector-valued function $b: \R^N \to \R^N$; see \cref{sec:problem_statement} for a precise statement. State equations with a similar structure appear for example in models of heat conduction, where $a$ stands for the heat conductivity and depends on the temperature $y$ (see, e.g., \cite{Bejan2013,ZeldovichRaizer1966}); the vector-valued function $b$ is a generalized advection term. The salient point, of course, is the non-differentiability of the conduction coefficient $a$ that allows for different behavior in different temperature regimes with sharp phase transitions but makes the analytic and numerical treatment challenging.

The study of optimal control problems for non-smooth partial differential equations is relatively recent, and previous works have focused on problems for non-smooth semilinear equations, see \cite{MeyerSusu2017,MeyerSusu2017} as well as the pioneering work \cite[Chap.~2]{Tiba1990}. To the best of our knowledge, this is the first work to treat non-smooth quasi-linear problems.

The main difficulty in the treatment of such control problems is the derivation of useful optimality conditions. Following \cite{MeyerSusu2017,Constantin2017,Tiba1990,Barbu1984,NeittaanmakiTiba1994}, we will use a regularization approach to approximate the original problem by corresponding regularized problems that allows obtaining regularized optimality conditions. By passing to the limit, we then derive so-called C-stationarity conditions involving Clarke's generalized gradient of the non-smooth term. However, unlike non-smooth semilinear equations, the main difficulty in our case is that the nonlinearity appears in the higher-order term, which presents an obstacle to directly passing to the limit in the regularized optimality systems.
To circumvent this, we will follow a dual approach, where instead of passing to the limit in the regularized adjoint equation, we shall do so in a linearized state equation and apply a duality argument.
We also derive strong stationarity conditions involving a sign condition on the adjoint state and show that under additional assumptions on the non-smooth nonlinearity (piecewise differentiability ($PC^1$) with countably many points of non-differentiability), both stationarity conditions coincide. Furthermore, in this case a relaxed optimality system can be derived that is amenable to numerical solution using a semi-smooth Newton method.

The paper is organized as follows. This introduction ends with some notations used throughout the paper. \cref{sec:problem_statement} then gives a precise statement of the optimal control problem together with the fundamental assumptions.
The following \cref{sec:quasilinear} is devoted to the directional differentiability and regularizations of the control-to-state operator. These will be used in \cref{sec:OS} to derive C- and strong stationarity conditions. \Cref{sec:PC1} then considers the special case that $a$ is a countably $PC^1$.
Numerical examples illustrating the solution of the relaxed optimality conditions are presented in \cref{sec:numerical_experiment}.

\paragraph*{Notations.}
By $\Id$, we denote the identity operator in a Banach space $X$. For a given point $u\in X$ and $\rho>0$, we denote by $B_X(u,\rho)$ and $\overline B_X(u,\rho)$, the open and closed balls, respectively, of radius $\rho$ centered at $u$.
For $u\in X$ and $\xi \in X^*$, the dual space of $X$, we denote by $\left \langle \xi, y \right\rangle$ their duality product. For Banach spaces $X$ and $Y$, the notation $X \hookrightarrow Y$ means that $X$ is continuously embedded in $Y$ and $X \Subset Y$ means that $X$ is compact embedded in $Y$.

If $K$ is a measurable subset in $\R^d$, the notation $|K|$ stands for the $d$-dimensional Lebesgue measure of $K$.
For a function $f:\Omega\to \R$ defined on a domain $\Omega\subset \R^d$ and $t\in \R$, the symbols $\{f \geq t\}$ and $\{f= t\}$ stand for the sets of a.e. $x \in \Omega$ such that $f(x) \geq t$ and $f(x) =t$, respectively. For any set $A\subset \Omega$, the symbol $\1_{A}$ denotes the indicator function of $A$, i.e., $\1_A(x) = 1$ if $x \in A$ and $\1_A(x) =0$ otherwise.

Finally, $C$ stands for a generic positive constant, which may be different at different places of occurrence. We also write, e.g., $C(\tau)$ for a constant depending only on the parameter $\tau$.

\section{Problem statement}\label{sec:problem_statement}

Let $\Omega$ be a bounded domain in $\R^N$, $N \geq 2$, with $C^{1}$-boundary $\partial\Omega$.
We consider for $p>\frac{N}2$ the optimal control problem
\begin{equation}\label{eq:P}
    \tag{P}
    \left\{
        \begin{aligned}
            \min_{u\in L^p(\Omega), y\in H^1_0(\Omega)} &J(y,u) \\
            \text{s.t.} \quad
            &\begin{aligned}[t]
                -\div [a(y)\nabla y + b(\nabla y)] &= u && \text{in } \Omega,\\
                y &= 0 && \text{on } \partial\Omega.
            \end{aligned}
        \end{aligned}
    \right.
\end{equation}

For the remainder of this paper, we make the following assumptions.
\begin{assumption}
\item \label{ass:quasi_func1}
    The function $a: \R \to \R$ is directionally differentiable, i.e., for any $y, h \in \R$ the limit
    \begin{equation*}
        a'(y;h) := \lim\limits_{t \to 0^+}\frac{a(y + th) - a(y)}{t}
    \end{equation*}
    exists. Moreover, $a$ satisfies
    \begin{equation*}
        a(y) \geq a_0 > 0 \quad \text{for all } y \in \R
    \end{equation*}
    for some constant $a_0$. In addition, for each $M>0$ there exists a constant $C_M>0$ such that
    \begin{equation*}
        |a(y_1) - a(y_2)| \leq C_M|y_1 - y_2| \quad \text{for all } y_i \in \R, |y_i| \leq M, i=1,2.
    \end{equation*}
\item \label{ass:quasi_func2}
    The function $b: \R^N \to \R^N$ is monotone, globally Lipschitz continuous, and continuously differentiable. Moreover, if $N\geq 3$, there exist constants $C_b > 0$ and $\sigma \in (0,1)$ such that
    \begin{equation*} \label{eq:add_regu}
        |b(\xi)| \leq C_b \left(1 + |\xi|^{\sigma} \right) \quad \text{for all} \quad \xi \in \R^N.
    \end{equation*}
\item \label{ass:cost_func}
    The cost function $J: H^1_0(\Omega) \times L^p(\Omega) \to \R$ is weakly lower semicontinuous and continuously Fr\'{e}chet differentiable. Furthermore, for any $M>0$, there exists a function $g_M: [0, \infty) \to \R$ such that $\lim_{t \to +\infty}g_M(t) = +\infty$ and for any $y \in H^1_0(\Omega) \cap C(\overline{\Omega})$ and $u \in L^p(\Omega)$ with $\left \|y \right\|_{H^1_0(\Omega)} + \left \|y \right\|_{C(\overline\Omega)} \leq M \left \|u \right\|_{L^p(\Omega)}$, it holds that
    \begin{equation*}
        J(y,u) \geq g_M \left(\left\|u \right\|_{L^p(\Omega)}\right).
    \end{equation*}
\end{assumption}

\begin{example}
    \cref{ass:quasi_func1} is satisfied, e.g., for the class of functions defined by
    \begin{equation*} 
        a(t) = a_0 + \sum_{i=1}^{k+1} \1_{(t_{i-1},t_i]}(t)a_i(t)  \quad \text{for all } t \in \R,
    \end{equation*}
    where $a_0 >0$, $-\infty =: t_0 <t_1 < t_2 < \cdots < t_k < t_{k+1} := \infty$ with the convention $(t_k, t_{k+1}] := (t_k, \infty)$, and $a_i$, $1 \leq i \leq k+1$, are nonnegative $C^1$-functions on $\R$ satisfying
    \begin{equation*}
        a_{i}(t_i) = a_{i+1}(t_i) \quad \text {for all} \quad 1 \leq i \leq k.
    \end{equation*}
    Note that such functions are continuous and piecewise continuously differentiable ($PC^1$) with a finite set of points of non-differentiability, see \cref{sec:PC1};
    explicit examples from this class are, e.g., $a(t) = 1+|t|$ (cf.~\cref{sec:semi-smooth}) or $a(t) = \max\{1,t\}$.

    An example that is not $PC^1$ is the following. Let 
    \begin{align*}
        f(t) &= 
        \begin{cases} 
            t^2 \sin(t^{-1}) & \text{if }t\neq 0,\\
            0 & \text{if }t= 0,\\
        \end{cases}
        \shortintertext{and}
        a(t) &= 1+\max\{f(t),0\}.
    \end{align*}
    Then $a$ is Lipschitz continuous (as the pointwise maximum of Lipschitz continuous functions) and directionally differentiable but not $PC^1$ since $f'$ is not continuous in $t=0$. (A more pathological example can be found in \cite[Ex.~5.3]{Kummer:2002}.)
\end{example}

\begin{example}
    Let $f$ be a $C^1$ function such that for some constants $C_1, C_2 >0$ and $\sigma \in (0,1)$,
    \begin{equation*}
        \label{eq:exam}
        |f(t)t| \leq C_1(1 + t^\sigma), \quad 0 \leq f(t) + f'(t)t \leq C_2, \quad\text{for all } t\geq 0.
    \end{equation*}
    Then the vector-valued function $b: \R^N \to \R^N$ given by $b(\xi) := f(|\xi|)\xi$ satisfies \cref{ass:quasi_func2}. For example, we can choose
    \begin{enumerate}[label=(\roman*)]
        \item $b(\xi) = \left(1+ |\xi| \right)^{r-2}\xi$ with $r \in (1,2]$ for $N=2$ and $r \in (1,2)$ for $N\geq 3$;
        \item $b(\xi) =(1+|\xi|^2)^{-1/2}\xi$.
    \end{enumerate}
\end{example}

\section{Properties of the control-to-state operator} \label{sec:quasilinear}

In this section, we derive the necessary results for the state equation
\begin{equation} \label{eq:state}
    \left\{
        \begin{aligned}
            -\div [a(y)\nabla y + b(\nabla y)]& = u && \text{in } \Omega, \\
            y &=0 && \text{on } \partial\Omega,
        \end{aligned}
    \right.
\end{equation}
as well as its regularization that are required in \cref{sec:OS}.

\subsection{Existence, uniqueness, and regularity of solutions to the state equation}

We first address existence and uniqueness of solutions to \eqref{eq:state}. Here and in the following, we always consider weak solutions. The following proof is based on the technique from \cite[Thm.~2.2]{CasasTroltzsch2009} with some modifications.
\begin{theorem} \label{thm:existence}
    Let $p^* > N$ be arbitrary.
    Assume that \cref{ass:quasi_func1,ass:quasi_func2} hold. Then, for any $u \in W^{-1,p^*}(\Omega)$, there exists a unique solution $y_u \in H^{1}_0(\Omega) \cap C(\overline{\Omega})$ to \eqref{eq:state} satisfying the a priori estimate
    \begin{equation} \label{eq:apriori}
        \|y_u\|_{H^{1}_0(\Omega)} + \| y_u\|_{C(\overline{\Omega})} \leq C_\infty \|u\|_{W^{-1,p^*}(\Omega)}
    \end{equation}
    for some constant $C_\infty>0$ depending only on $a_0, p^*$, $N$, and $|\Omega|$.
\end{theorem}
\begin{proof}
    To show existence, fix $M > 0$ and define the truncated coefficient
    \begin{equation*}
        a_M(y):=
        \begin{cases}
            a(M) &\text{if } y>M,\\
            a(y) &\text{if } |y|\leq M,\\
            a(-M) &\text{if } y<-M.
        \end{cases}
    \end{equation*}
    Fixing $z \in L^2(\Omega)$, we consider the nonlinear but smooth elliptic equation
    \begin{equation} \label{eq:state2}
        \left\{
            \begin{aligned}
                -\div [a_M(z)\nabla y + b(\nabla y)] &= u && \text{in } \Omega, \\
                y &=0 && \text{on } \partial\Omega.
            \end{aligned}
        \right.
    \end{equation}
    Define the mapping $F^z_M: H^1_0(\Omega) \to H^{-1}(\Omega)$ given by
    \begin{equation*}
        \left\langle F^z_M(y), w \right\rangle := \int_{\Omega} \left[ a_M(z(x)) \nabla y + b(\nabla y) - b(0) \right ]\cdot \nabla w dx, \quad y, w \in H^1_0(\Omega).
    \end{equation*}
    Due to the continuity and the monotonicity of $b$ and \cref{ass:quasi_func1}, $F^z_M$ is coercive, strongly monotone, and hemicontinuous. Since $p^* > N \geq 2$, we have $W^{-1,p^*}(\Omega) \hookrightarrow H^{-1}(\Omega)$ and so $u \in H^{-1}(\Omega)$.
    Then, \cite[Thm.~26.A]{Zeidler2B} implies that there exists a unique $y^z_M \in H^1_0(\Omega)$ which satisfies equation \eqref{eq:state2}. The strong monotonicity of $F^z_M$ thus implies that
    \begin{equation} \label{eq:apriori1}
        \|\nabla y^z_M \|_{L^2(\Omega)} \leq \frac{1}{a_0} \| u\|_{H^{-1}(\Omega)}
        \leq \frac{1}{a_0} C(p^*,N, \Omega) \| u\|_{W^{-1,p^*}(\Omega)}.
    \end{equation}
    Due to the mean value theorem, we obtain
    \begin{equation}
        b(\nabla y^z_M(x)) - b(0) = \int_0^1 J_b(t\nabla y^z_M(x)) \nabla y^z_M(x)dt \quad \text{for a.e. }x\in \Omega,
    \end{equation}
    where $J_b$ stands for the Jacobian matrix of $b$. Since $b$ is globally Lipschitz continuous, there exists a constant $L_b>0$ such that
    $ |J_b(\xi)| \leq L_b$ for all $ \xi \in \R^N$.
    Setting
    \begin{equation*}
        T^z_M(x) := \int_0^1 J_b(t\nabla y^z_M(x))dt
    \end{equation*}
    yields that $T^z_M$ is non-negative definite and
    $ |T^z_M(x)| \leq L_b$ for a.a. $x \in \Omega$.
    Due to \eqref{eq:state2}, $y^z_M$ satisfies
    \begin{equation*}
        \left\{
            \begin{aligned}
                -\div [\hat a^z_M\nabla y] & = u && \text{in } \Omega, \\
                y &=0 && \text{on } \partial\Omega,
            \end{aligned}
        \right.
    \end{equation*}
    where $\hat a^z_M(x) := a_M(z(x))\Id + T^z_M(x)$. It is easy to see for a.e. $x \in \Omega$ that
    \begin{equation*}
        a_0 |\xi|^2 \leq \hat a^z_M(x)\xi \cdot \xi \leq \left(\max\left\{a_M(t): |t| \leq M \right \} + L_b\right)|\xi|^2 \quad \text{for all} \quad \xi \in \R^N.
    \end{equation*}
    It follows from the fact $u \in W^{-1,p^*}(\Omega)$ and \cite[Thm.~3.10]{Adams} that 
    \begin{equation}
        \label{eq:Lp-repr}
        u = u_0 + \sum_{i=1}^{N} \partial_{x_i}u_i
    \end{equation}
    for some functions $u_i \in L^{p^*}(\Omega)$, $i=0,1,\dots,N$. Moreover, one has
    \begin{equation*}
        \| u \|_{W^{-1,p^*}(\Omega)} = \inf \left\{ \sum_{i=0}^{N}  \| u_i \|_{L^{p^*}(\Omega)}: u_0,\dots, u_N \ \text {satisfy \eqref{eq:Lp-repr}}   \right\}.
    \end{equation*}
    Obviously, $u_0 \in L^{q^*}(\Omega)$ with $q^* := \frac{Np^*}{N +p^*} < p^*$.
    Since $p^* > N$, the Stampacchia theorem \cite[Thm.~12.4]{Chipot2009} implies that there exists a constant $c_\infty : = c_\infty(a_0, p^*, N, |\Omega|)$ such that    
    \begin{align*} 
        \|y^z_M \|_{L^\infty(\Omega)} & \leq c_\infty \left(  \| u_0 \|_{L^{q^*}(\Omega)} + \sum_{i=1}^{N}  \| u_i \|_{L^{p^*}(\Omega)} \right) \\
        & \leq  c_\infty  \sum_{i=0}^{N}  \| u_i \|_{L^{p^*}(\Omega)}.
    \end{align*}
    As the above estimate holds for all families $\{u_i\}_{0 \leq i \leq N} \subset L^{p^*}(\Omega)$ which satisfy \eqref{eq:Lp-repr}, there holds
    \begin{equation} \label{eq:apriori2}
        \|y^z_M \|_{L^\infty(\Omega)} \leq c_\infty \| u \|_{W^{-1,p^*}(\Omega)}.
    \end{equation}
    Besides, the continuity of $y^z_M$ follows as usual; see, for instance, \cite[Thm.~8.29]{Gilbarg_Trudinger}. Combining this with estimates \eqref{eq:apriori1} and \eqref{eq:apriori2} yields
    \begin{equation} \label{eq:apriori3}
        \|y^z_M \|_{H^1_0(\Omega)} + \|y^z_M \|_{C(\overline\Omega)} \leq c(a_0, p^*, N, |\Omega|) \| u \|_{W^{-1,p^*}(\Omega)}.
    \end{equation}
    We now prove that $y^z_M$ is a solution to \eqref{eq:state} for $M$ large enough. To this end, we define the mapping $ F_M : L^2(\Omega) \ni z \mapsto y^z_M \in L^2(\Omega)$.
    Let us take $z_n \to z$ in $L^2(\Omega)$ and set $y_n := F_M(z_n)$, $y := F_M(z)$. We have
    \begin{equation}
        \left\{
            \begin{aligned}
                -\div [a_M(z_n)\nabla y_n + b(\nabla y_n)] & = u && \text{in } \Omega, \\
                y_n &=0 && \text{on } \partial\Omega
            \end{aligned}
        \right.
    \end{equation}
    and
    \begin{equation}
        \left\{
            \begin{aligned}
                -\div [a_M(z)\nabla y + b(\nabla y)] &= u && \text{in } \Omega, \\
                y & = 0 && \text{on } \partial\Omega.
            \end{aligned}
        \right.
    \end{equation}
    Subtracting these two equations yields that
    \begin{equation*}
        \left \{
            \begin{aligned}
                -\div \left [a_M(z_n)\left(\nabla y_n - \nabla y\right) + b(\nabla y_n) - b(\nabla y)\right] &= \div \left[ \left(a_M(z_n) - a_M(z) \right) \nabla y \right] && \text{in } \Omega, \\
                y_n - y &=0 && \text{on } \partial\Omega.
            \end{aligned}
        \right.
    \end{equation*}
    By multiplying the above equation with $y_n -y$, integration over $\Omega$, and then using the monotonicity of $b$, we have
    \begin{equation} \label{eq:esti}
        a_0 \| \nabla y_n - \nabla y\|_{L^2(\Omega)} \leq \| \left(a_M(z_n) - a_M(z) \right) \nabla y\|_{L^2(\Omega)}.
    \end{equation}
    By virtue of the Lebesgue dominated convergence theorem, the right hand side of \eqref{eq:esti} tends to zero as $n \to \infty$. Consequently, $y_n$ converges to $y$ in $H^1_0(\Omega)$. It follows that $F_M$ is continuous as a function from $L^2(\Omega)$ to $H^1_0(\Omega)$. On the other hand, due to the compact embedding $H^1_0(\Omega) \Subset L^2(\Omega)$, $F_M$ is a compact operator. As a result of \eqref{eq:apriori3}, the range of $F_M$ is therefore contained in a ball in $L^2(\Omega)$.
    The Schauder fixed-point theorem guarantees the existence of a function $y_M \in L^2(\Omega)$ satisfying $y_M = F_M(y_M)$.

    Choosing now $M \geq c(a_0, p, N, |\Omega|) \| u \|_{W^{-1,p^*}(\Omega)}$, it follows from \eqref{eq:apriori3} that $\| y_M \|_{C(\overline{\Omega})} \leq M$ and so $a_M(y_M(x)) = a(y_M(x))$ for all $x \in \Omega$. Therefore, $y_M$ solves \eqref{eq:state}.

    \medskip

    To show the uniqueness of the solution, assume that $y_1$ and $y_2$ are two solutions to \eqref{eq:state} in $H^1_0(\Omega) \cap C(\overline{\Omega})$. Let us define, for any $\epsilon >0$, the open sets
    \begin{equation*}
        K_0 := \left\{ x \in \Omega \mid y_2(x) > y_1(x)\right\}
        \quad \text{ and }\quad
        K_\epsilon := \left\{ x \in \Omega \mid y_2(x) > \epsilon +y_1(x)\right\}.
    \end{equation*}
    We set $z_\epsilon(x) := \min\{\epsilon, (y_2(x)-y_1(x))^+\}$, where $t^+ := \max(t,0)$. We then have $z_\epsilon \in H^{1}_0(\Omega)$, $|z_\epsilon| \leq \epsilon$, $z_\epsilon = \epsilon$ on $ K_\epsilon$, and $\nabla z_\epsilon = \1_{ K_0 \setminus K_\epsilon} \nabla (y_2 - y_1) $.

    Multiplying the equations corresponding to $y_i$ by $z_\epsilon$, integrating over $\Omega$, and using integration by parts, we have
    \begin{equation*}
        \int_{\Omega} \left[a(y_i) \nabla y_i + b(\nabla y_i) \right]\cdot \nabla z_\epsilon dx = \langle u, z_\epsilon \rangle, \quad i=1,2.
    \end{equation*}
    Subtracting these equations yields that
    \begin{equation*}
        \int_\Omega a(y_2) |\nabla z_\epsilon|^2 + (b(\nabla y_2) - b(\nabla y_1))\cdot \nabla z_\epsilon dx = \int_{\Omega} \left( a(y_1) -a(y_2) \right) \nabla y_1 \cdot \nabla z_\epsilon dx.
    \end{equation*}
    From this, the monotonicity of $b$, and \cref{ass:quasi_func1}, we obtain
    \begin{equation*}
        \begin{aligned}[t]
            a_0 \|\nabla z_\epsilon\|_{L^2(\Omega)}^2
            & \leq
            \int_\Omega a(y_2) |\nabla z_\epsilon|^2 dx + \int_{K_0 \setminus K_\epsilon} (b(\nabla y_2) - b(\nabla y_1))\cdot \nabla z_\epsilon dx \\
            & = \int_\Omega a(y_2) |\nabla z_\epsilon|^2 + (b(\nabla y_2) - b(\nabla y_1))\cdot \nabla z_\epsilon dx \\
            & = \int_{K_0 \setminus K_\epsilon} \left( a(y_1) -a(y_2) \right) \nabla y_1 \cdot \nabla z_\epsilon dx \\
            & \leq \|a(y_1)-a(y_2)\|_{L^\infty(K_0 \setminus K_\epsilon)} \|\nabla y_1\|_{L^2(K_0 \setminus K_\epsilon)}\|\nabla z_\epsilon\|_{L^2(K_0 \setminus K_\epsilon)}\\
            & \leq C_M \|y_1-y_2\|_{L^\infty(K_0 \setminus K_\epsilon)}\|\nabla y_1\|_{L^2(K_0 \setminus K_\epsilon)}\|\nabla z_\epsilon\|_{L^2(\Omega)} \\
            & \leq C_M \epsilon \|\nabla y_1\|_{L^2(K_0 \setminus K_\epsilon)}\|\nabla z_\epsilon\|_{L^2(\Omega)}.
        \end{aligned}
    \end{equation*}
    Here $M := \max\{ |y_i(x)| \mid x \in \overline \Omega, i =1,2 \}$. Combining this with the Poincar\'{e} inequality, we have
    \begin{equation*}
        \| z_\epsilon\|_{L^2(\Omega)} \leq C \epsilon \|\nabla y_1\|_{L^2(K_0 \setminus K_\epsilon)}
    \end{equation*}
    for some constant $C$. Since $\lim_{\epsilon \to 0} |K_0 \setminus K_\epsilon| = 0$ and $z_\epsilon = \epsilon$ on $K_\epsilon$,
    \begin{equation*}
        |K_\epsilon| = \epsilon^{-2}\int_{K_\epsilon} z_\epsilon^2dx \leq C^2\|\nabla y_1\|_{L^2(K_0 \setminus K_\epsilon)}^2 \to 0
    \end{equation*}
    as $\epsilon \to 0$. This implies that $|K_0| = \lim_{\epsilon \to 0}|K_\epsilon| = 0$. Consequently, $y_1 \geq y_2$ a.e. in $\Omega$.

    In the same way, we have $y_2 \geq y_1$ a.e. in $\Omega$ and hence $y_1 = y_2$.
\end{proof}

From now on, for each $u \in W^{-1,p^*}(\Omega)$, $p^* >N$, we denote by $y_u$ the unique solution to \eqref{eq:state}. The control-to-state operator $ W^{-1,p^*}(\Omega) \ni u \mapsto y_u \in H^1_0(\Omega)$ is denoted by $S$.

The following theorem on the regularity of solutions to equation \eqref{eq:state} will be crucial in proving the directional differentiability of the control-to-state operator $S$.
\begin{theorem} \label{thm:regu}
    Assume that \cref{ass:quasi_func1,ass:quasi_func2} are valid. Let $U$ be a bounded set in $W^{-1,p*}(\Omega)$ with $p^*>N$. Then, there exists a constant $s := s(a_0, \Omega,N,p^*, U) > N$ such that the following assertions hold:
    \begin{itemize}
        \item[(i)] If $u \in U$, then $y_u \in W^{1,s}_0(\Omega)$ and
            \begin{equation} \label{eq:regu}
                \|y_u\|_{W^{1,s}_0{(\Omega)}} \leq C_1
            \end{equation}
            for some constant $C_1>0$ depending only on $a_0, \Omega, N, p^*$, and $U$.
        \item[(ii)] If $u_n \to u$ in $W^{-1,p^*}(\Omega)$, then $y_{u_n} \to y_u$ in $W^{1,r}_0(\Omega) \cap C(\overline{\Omega})$ for all $1 \leq r < s$.
    \end{itemize}
\end{theorem}
\begin{proof}
    \emph{Ad (i):}
    Let $U$ be a bounded subset in $W^{-1,p^*}(\Omega)$. Due to the a priori estimate \eqref{eq:apriori}, there is a constant $M := M(a_0, p^*, N, \Omega,U)$ such that
    \begin{equation*}
        \|y_u \|_{C(\overline\Omega)} \leq M \quad \text{for all} \quad u \in U.
    \end{equation*}
    Fixing $u \in U$ and using the mean value theorem, we have
    \begin{equation*}
        b(\nabla y_u(x)) - b(0) = \left(\int_0^1 J_b(t\nabla y_u(x)) dt \right)\cdot \nabla y_u(x) \quad \text{for a.e. }x\in \Omega,
    \end{equation*}
    where $J_b$ again denotes the Jacobian matrix of $b$.
    Setting
    $ T(x) := \int_0^1 J_b(t\nabla y_u(x))dt$,
    we see that $T(x)$ is non-negative definite and
    $|T(x)| \leq L_b$ for a.a. $x \in \Omega$ with some constant $L_b$. Then, $y_u$ satisfies
    \begin{equation} \label{eq:regu1}
        \left\{
            \begin{aligned}
                -\div [\hat a(x) \nabla y_u] &= u && \text{in } \Omega, \\
                y_u & =0 && \text{on } \partial\Omega
            \end{aligned}
        \right.
    \end{equation}
    with $\hat a(x) := a(y_u(x))\Id + T(x)$. As above, we see that
    \begin{equation*}
        a_0 |\xi|^2 \leq \hat a(x) \xi \cdot \xi \leq \hat M |\xi|^2 \quad \text{for all } \xi \in \R^N \ \text{and for a.e. } x \in \Omega
    \end{equation*}
    with $\hat M:= C_MM + |a(0)| + L_b$.
    The regularity of solutions to \eqref{eq:regu1} (see, e.g.
    \cite[Thm.~2.1]{Bensoussan2002}) implies that there exists a constant $\delta :=\delta(a_0, \hat M, \Omega, N, p) > 0$ such that $y_u \in W^{1,\tau}_0(\Omega)$ for
    any $2 \leq \tau < 2 + \delta$. Moreover, it holds that
    \begin{equation} \label{eq:esti0}
        \|y_u \|_{W^{1,\tau}_0(\Omega)} \leq c_\tau \| u\|_{W^{-1,p^*}(\Omega)}
    \end{equation} for all $2 \leq \tau < 2 + \delta$ and for some constant $c_\tau$ depending only on $\tau$. For $N=2$, assertion (i) then follows from the Sobolev embedding.

    It remains to prove assertion (i) for the case where $N\geq 3$. For this, we rewrite equation \eqref{eq:state} in the form
    \begin{equation} \label{eq:regu2}
        \left\{
            \begin{aligned}
                -\div [a(y_u(x)) \nabla y_u] & = u + \hat u && \text{in } \Omega, \\
                y_u & =0 && \text{on } \partial\Omega,
            \end{aligned}
        \right.
    \end{equation}
    where $\hat u := \div[b(\nabla y_u)]$. We now use the Kirchhoff transformation
    $K(t) := \int_0^t a(\varsigma)d\varsigma$ (see \cite[Chap.~V]{Visintin}).
    By setting $\theta(x) := K(y_u(x))$ for $x \in \overline\Omega$, \eqref{eq:regu2} can be rewritten as follows
    \begin{equation} \label{eq:regu3}
        \left\{
            \begin{aligned}
                -\Delta \theta &= u + \hat u && \text{in } \Omega, \\
                \theta &=0 && \text{on } \partial\Omega.
            \end{aligned}
        \right.
    \end{equation}
    We then have
    \begin{equation} \label{eq:esti1}
        \|\theta\|_{L^1(\Omega)} \leq C(\Omega) \|\theta\|_{L^2(\Omega)} \leq C(\Omega) \|\nabla \theta\|_{L^2(\Omega)} \leq C(\Omega)\| u + \hat u\|_{H^{-1}(\Omega)}.
    \end{equation}
    Fixing $\tau \in (2, 2+\delta)$, we see from \eqref{eq:esti0} that $\nabla y_u \in \left( L^\tau(\Omega) \right)^N$. From this and \cref{ass:quasi_func2}, we can conclude that $b(\nabla y_u) \in \left( L^{\tau_1}(\Omega) \right)^N$, $\tau_1 := \tau/\sigma$ and
    \begin{equation*}
        |b(\nabla y_u(x))|^{\tau_1} \leq C_b^{\tau_1}2^{\tau_1-1}\left(1 + |\nabla y_u(x)|^{\tau} \right) \quad \text{for a.e. } x \in \Omega.
    \end{equation*}
    Consequently, we have $\hat u \in W^{-1,\tau_1}(\Omega)$ and
    \begin{equation} \label{eq:esti3}
        \begin{aligned}[t]
            \|\hat u\|_{W^{-1,\tau_1}(\Omega)} & 	\leq C \|b(\nabla y_u)\|_{L^{\tau_1}(\Omega)} \\
            & \leq C(\Omega, \tau)\left(1 + \|\nabla y_u\|^\tau_{L^{\tau}(\Omega)} \right)^{\sigma/\tau} \\
            & \leq C(\Omega, \tau)\left(1 + \| u\|_{W^{-1,p^*}(\Omega)}^\tau \right)^{\sigma/\tau} \\
            & \leq C(\Omega, \tau)\left(1 + \|u\|^\sigma_{W^{-1,p^*}(\Omega)} \right).
        \end{aligned}
    \end{equation}
    Here we have used the estimate \eqref{eq:esti0} and the inequality $(r_1+r_2)^{d} \leq r_1^d + r_2^d$ for $r_1,r_2 \geq 0$, $0 < d<1$.
    Consequently, $u + \hat u \in W^{-1,q}(\Omega)$ with $q := \min\{p^*,\tau_1\}$.
    We now apply \cite[Thms.~5.5.4' and 5.5.5']{Morrey1966} for problem \eqref{eq:regu3} to obtain $\theta \in W^{1,q}_0(\Omega)$ and
    \begin{equation*}
        \|\theta\|_{W^{1,q}_0(\Omega)} \leq C(N,q, \Omega) \left( \| u + \hat u\|_{W^{-1,q}(\Omega)} + \|\theta\|_{L^1(\Omega)} \right).
    \end{equation*}
    Combining this with \eqref{eq:esti1} yields
    \begin{equation*}
        \begin{aligned}
            \|\theta\|_{W^{1,q}_0(\Omega)} & \leq C(N,q, \Omega)\left( \| u + \hat u\|_{W^{-1,q}(\Omega)} + \|u + \hat u\|_{H^{-1}(\Omega)} \right)\\
            & \leq C(N,p^*,\tau, \Omega) \left( \| u \|_{W^{-1,p^*}(\Omega)} + \|\hat u\|_{W^{-1,\tau_1}(\Omega)} \right),
        \end{aligned}
    \end{equation*}
    which together with \eqref{eq:esti3} gives
    \begin{equation*}
        \|\theta\|_{W^{1,q}_0(\Omega)} \leq C(N,p^*,\tau, \Omega) \left(1+ \| u \|_{W^{-1,p^*}(\Omega)} + \| u \|_{W^{-1,p^*}(\Omega)}^{\sigma} \right)	.
    \end{equation*}
    Since $y_u = K^{-1}(\theta)$, we obtain $\nabla y_u = \frac{1}{a(K^{-1}(\theta))} \nabla \theta$. We then have
    \begin{equation*}
        \begin{aligned}[t]
            \|\nabla y_u\|_{L^{q}(\Omega)} & \leq \frac{1}{a_0}\|\nabla \theta\|_{L^{q}(\Omega)} \\
            &\leq C(a_0, \Omega, N,p^*,U) \left(1+ \| u \|_{W^{-1,p^*}(\Omega)} + \| u \|_{W^{-1,p^*}(\Omega)}^{\sigma} \right).
        \end{aligned}
    \end{equation*}
    We now distinguish the following cases.
    \begin{enumerate}[label={Case }\arabic*:, align=left]
        \item $q>N$. By setting $s:= q$, we obtain \eqref{eq:regu}.

        \item$q \leq N$. In this case, one has $q = \tau_1 < p^*$. By the same arguments as above, we can deduce that $\nabla y_u \in L^{q_2}(\Omega)$ with
            $q_2 := \min\{ \tau_2, p^* \}$ for $\tau_2 := \frac{\tau_1}{\sigma} = \frac{\tau}{\sigma^2}$ and
            \begin{equation*} \label{eq:regu5}
                \|\nabla y_u\|_{L^{q_2}(\Omega)} \leq C(a_0, \Omega, N,p^*,U) \left(1+ \| u \|_{W^{-1,p^*}(\Omega)} + \| u \|_{W^{-1,p^*}(\Omega)}^{\sigma} + \| u \|_{W^{-1,p^*}(\Omega)}^{2\sigma} \right).
            \end{equation*}
            We now choose the smallest integer $k \geq 2$ such that $\tau_k \geq p^*$ with $\tau_k := \frac{\tau}{\sigma^k}$.
            Proceeding by induction, we obtain $\nabla y_u \in L^{q_k}(\Omega)$ with $q_k := \min \{\tau_k, p^* \} = p^*$ and
            \begin{equation*}
                \|\nabla y_u\|_{L^{p^*}(\Omega)} \leq C(a_0, \Omega, N,p^*,U) \left(\| u \|_{W^{-1,p^*}(\Omega)} +\sum_{i=0}^k \| u \|_{W^{-1,p^*}(\Omega)}^{i\sigma} \right).
            \end{equation*}
            In this case, we set $s := p^*$ and then obtain estimate \eqref{eq:regu}.
    \end{enumerate}

    \medskip

    \emph{Ad (ii):}
    Assume that $u_n \to u$ in $W^{-1,p^*}(\Omega)$. We set $y_n := S(u_n)$. From the convergence of $\{u_n\}$ and the a priori estimate \eqref{eq:regu}, we deduce that $\{y_n\}$ is bounded in $W^{1,s}_0(\Omega)$. By passing to a subsequence, we can assume that $y_n \rightharpoonup y$ in $W^{1,s}_0(\Omega)$ and $y_n \to y$ in $C(\overline{\Omega})$ as $n \to \infty$ for some $y \in W^{1,s}_0(\Omega)$. Setting $M:= \max\{\|y\|_{C(\overline{\Omega})}, \|y_n\|_{C(\overline{\Omega})}\}$, we obtain from \cref{ass:quasi_func1} that
    \begin{equation*}
        \|a(y_n)- a(y)\|_{C(\overline\Omega)} \leq C_M \|y_n - y\|_{C(\overline{\Omega})} \to 0 \quad \text{as} \quad n \to \infty
    \end{equation*}
    for some constant $C_M > 0$.
    For $n \geq 1$, we have
    \begin{equation*}
        \left\{
            \begin{aligned}
                -\div \left [(a(y_n) - a_0 /2) \nabla y_n \right] + B( y_n) & = u_n && \text{in } \Omega, \\
                y_n & =0 && \text{on } \partial\Omega
            \end{aligned}
        \right.
    \end{equation*}
    with $B(z):= -\div [ a_0/2 \nabla z + b(\nabla z)]$. The global Lipschitz continuity of $b$ and the boundedness of $\{y_n\}$ in $H^1_0(\Omega)$ ensure the boundedness of $\{B(y_n)\}$ in $H^{-1}(\Omega)$. There exists a subsequence of $\{B(y_n)\}$, denoted in the same way, such that $B(y_n) \rightharpoonup \psi$ in $H^{-1}(\Omega)$.
    Letting $n \to \infty$ yields
    \begin{equation*}
        \left\{
            \begin{aligned}
                -\div \left [(a(y) - a_0 /2) \nabla y \right] + \psi &= u && \text{in } \Omega, \\
                y &=0 && \text{on } \partial\Omega.
            \end{aligned}
        \right.
    \end{equation*}
    We have
    \begin{multline*}
        \liminf_{n \to \infty} \left\langle B(y_n), y_n \right\rangle \\
        \begin{aligned}
            & = \liminf_{n \to \infty} \left\langle u_n +\div \left [(a(y_n) - a_0 /2) \nabla y_n \right] , y_n \right\rangle \\
            & =\liminf_{n \to \infty} \left\langle u_n, y_n \right\rangle - \limsup_{n \to \infty} \int_{\Omega} (a(y_n)-a_0/2)|\nabla y_n|^2dx \\
            & \leq \left\langle u, y \right\rangle - \liminf_{n \to \infty} \int_{\Omega} (a(y_n)-a_0/2)|\nabla y_n|^2dx \\
            & \leq \left\langle u, y \right\rangle - \liminf_{n \to \infty} \int_{\Omega} (a(y)-a_0/2)|\nabla y_n|^2dx + \lim_{n \to \infty} \| a(y_n) -a(y) \|_{C(\overline\Omega)} \| \nabla y_n\|^2_{L^2(\Omega)} \\
            & \leq \left\langle u, y \right\rangle - \int_{\Omega} (a(y)-a_0/2)|\nabla y|^2dx \\
            & = \left\langle \psi, y \right\rangle.
        \end{aligned}
    \end{multline*}
    Due to \cref{lem:mono}, $B$ is maximally monotone as an operator from $H^1_0(\Omega)$ to $H^{-1}(\Omega)$. From the strong-to-weak closedness of maximally monotone operators (see, e.g., \cite[Lemma 5.1, Chap-~XI]{Visintin}), we obtain that $\psi = B(y)$.

    Consequently, $y_n$ and $y$ satisfy the equation
    \begin{equation*}
        \left \{
            \begin{aligned}
                -\div \left [a(y_n)\left(\nabla y_n - \nabla y\right) + b(\nabla y_n) - b(\nabla y)\right] &= u_n - u + \div \left[ \left(a(y_n) - a(y) \right) \nabla y \right] && \text{in } \Omega, \\
                y_n - y &=0 && \text{on } \partial\Omega.
            \end{aligned}
        \right.
    \end{equation*}
    Multiplying the above equation by $(y_n -y)$, integration over $\Omega$, and using \cref{ass:quasi_func1,ass:quasi_func2} gives
    \begin{equation} \label{eq:conv}
        a_0 \| \nabla ( y_n - y )\|^2_{L^2(\Omega)} \leq \left \langle u_n -u, y_n - y \right \rangle - \int_{\Omega} \left(a(y_n) - a(y) \right) \nabla y \cdot (\nabla y_n - \nabla y)dx.
    \end{equation}
    Since the embedding $W^{-1,p^*}(\Omega) \hookrightarrow H^{-1}(\Omega)$ is continuous, it follows that $u_n \to u$ in $H^{-1}(\Omega)$. Consequently, the right hand side of \eqref{eq:conv} tends to zero. We then have $y_n \to y$ strongly in $H^1_0(\Omega)$. Therefore, $\nabla y_n \to \nabla y$ in measure. The convergence of $\nabla y_n$ to $\nabla y$ in $\left(L^r(\Omega) \right)^N$, $1 \leq r < s$, follows from \cite[Chap.~XI, Prop.~3.10]{Visintin}. This and the uniqueness of solutions to \eqref{eq:state} yield assertion (ii).
\end{proof}

As a direct consequence of \cref{thm:regu}, we have
\begin{corollary} \label{cor:continuity}
    Let $p^*>N$ be arbitrary. Assume that \cref{ass:quasi_func1,ass:quasi_func2} are satisfied. Then, the operator $S: W^{-1,p^*}(\Omega) \to H^1_0(\Omega) \cap C(\overline{\Omega})$ is continuous.
\end{corollary}

Since $p > N/2$, we can choose a number $\tilde{p} > N$ such that
\begin{equation} \label{eq:p_star}
    \frac{1}{p} < \frac{1}{\tilde p} + \frac{1}{N}.
\end{equation}
This implies that the embedding $L^p(\Omega) \Subset W^{-1,\tilde p}(\Omega)$ is compact.

Let $\bar u \in L^p(\Omega)$ be arbitrary, but fixed and $\bar \rho >0$ be a constant. From now on, we fix
\begin{equation} \label{eq:U_set}
    U := B_{L^p(\Omega)}(\bar u, 2 \bar \rho) \quad \text{and} \quad \bar s \in (N, s)
\end{equation}
with $s$ as given in \cref{thm:regu} corresponding to $p^* := \tilde{p}$ and $U$. Let us define constants $s_1$ and $s_2$ such that
\begin{equation} \label{eq:s1}
    s_1
    \begin{cases}
        > 2 \quad & \text{if } N=2,\\
        \in \left(\frac{2\bar s}{\bar s-2}, \frac{2N}{N-2}\right) & \text{if } N\geq 3,
    \end{cases} 
    \qquad \text{and} \qquad \frac{1}{\bar s} + \frac{1}{s_1} + \frac{1}{s_2} = \frac{1}{2}.
\end{equation}
Note that the embedding $H^1_0(\Omega) \Subset L^{s_1}(\Omega)$ is compact. The following property of $S$ is a direct consequence of \cref{thm:regu} and the compact embedding $L^p(\Omega) \Subset W^{-1, \tilde p}(\Omega)$.
\begin{corollary} \label{cor:continuity2}
    Assume that \cref{ass:quasi_func1,ass:quasi_func2} are satisfied. Then, the operator $S: U \to W^{1,\bar s}_0(\Omega)$ is continuous and completely continuous, i.e., $u_n\weakto u$ implies that $S(u_n)\to S(u)$.
\end{corollary}

\subsection{Directional differentiability of the control-to-state operator}

In order to derive stationarity conditions for problem \eqref{eq:P}, we require directional differentiability of the control-to-state operator $S$.
We thus consider for given $y \in H^1_0{(\Omega)}$ the ``linearized'' equation
\begin{equation} \label{eq:dire_deri}
    \left\{
        \begin{aligned}
            -\div \left [\left(a(y)\Id + J_b(\nabla y) \right) \nabla z + a'(y; z) \nabla y \right] & = v && \text{in } \Omega, \\
            z &=0 && \text{on } \partial\Omega.
        \end{aligned}
    \right.
\end{equation}
We begin with a technical lemma regarding the directional derivatives of the nonsmooth nonlinearity $a$.
\begin{lemma}
    \label{lem:Lip-Hadamard}
    Let $M >0$ and $y \in C(\overline\Omega)$ such that $|y(x)| \leq M$ for all $x \in \overline\Omega$. Under \cref{ass:quasi_func1}, there hold:
    \begin{itemize}
        \item[(i)] For all $x \in \overline\Omega$ and $h \in \R$,
            \begin{equation*}
                \lim\limits_{\substack{t \to 0^+ \\ h' \to h }}\frac{a(y(x) + th')- a(y(x))}{t} = a'(y(x);h);
            \end{equation*}
        \item[(ii)] For all $x \in \overline\Omega$, the mapping $\R \ni \eta \mapsto a'(y(x);\eta) \in \R$ is Lipschitz continuous with Lipschitz constant $C_{2M}$ as given in \cref{ass:quasi_func1}.
    \end{itemize}
\end{lemma}	
\begin{proof}
    Let us fix $x \in \overline\Omega$ and define the function $a_M: (-2M,2M) \to \R$ given by $a_M(t) = a(t)$ for all $t \in (-2M,2M)$. By virtue of \cref{ass:quasi_func1}, $a_M$ is directionally differentiable and Lipschitz continuous with Lipschitz constant $C_{2M}$. Thanks to \cite[Prop.~2.49]{Bonnans_Shapiro}, for all $\eta \in (-2M,2M)$,  $a_M$ is directionally differentiable at $\eta$ in the Hadamard sense, i.e.,
    \begin{equation*}
        \lim\limits_{\substack{t \to 0^+ \\ h' \to h }}\frac{a_M(\eta + th')- a(\eta)}{t} = a_M'(\eta;h) \quad \text {for all } h \in \R,
    \end{equation*}
    which together with $y(x) \in (-2M,2M)$ gives assertion (i). Furthermore, \cite[Prop.~2.49]{Bonnans_Shapiro} implies that $a_M'(y(x);\cdot)$ is Lipschitz continuous with Lipschitz constant $C_{2M}$ on $\R$. From this and the fact that $a'(y(x); \cdot) = a_M'(y(x); \cdot)$, (ii) is thus derived.  
\end{proof}

We now show existence and uniqueness of solutions to \eqref{eq:dire_deri}.
\begin{theorem} \label{thm:dire_deri1}
    Let \cref{ass:quasi_func1,ass:quasi_func2} hold and let $\bar s$ be given as in \eqref{eq:U_set}. Assume that $y \in W^{1,\bar s}_0(\Omega)$. Then, for each $v \in H^{-1}(\Omega)$, equation \eqref{eq:dire_deri} admits a unique solution $z \in H^1_0(\Omega)$.
\end{theorem}
\begin{proof}
    Here we first prove the uniqueness and then the existence of solutions. The arguments to show the uniqueness are similar to the ones in the proof of \cref{thm:existence}.

    \medskip

    \emph{Step 1: Uniqueness of solutions.} Let $z_1$ and $z_2$ be two solutions to \eqref{eq:dire_deri} in $H^1_0(\Omega)$.
    We define the measurable sets
    \begin{equation*}
        K_0 := \left\{ x \in \Omega \mid z_2(x) > z_1(x)\right\}, \quad K_\epsilon := \left\{ x \in \Omega \mid z_2(x) > \epsilon +z_1(x)\right\}, \quad \epsilon >0.
    \end{equation*}
    We set $z_\epsilon(x) := \min\{\epsilon, (z_2(x)-z_1(x))^+\}$. Then, $z_\epsilon \in H^{1}_0(\Omega)$, $|z_\epsilon| \leq \epsilon$, $z_\epsilon = \epsilon$ on $K_\epsilon$, and $\nabla z_\epsilon = \1_{ K_0 \setminus K_\epsilon}\nabla (z_2 - z_1)$.
    Multiplying the equations corresponding to $z_i$ by $z_\epsilon$, integrating over $\Omega$, and using integration by parts yields
    \begin{equation*}
        \int_{\Omega} \left[\left(a(y) + J_b(\nabla y) \right) \nabla z_i + a'(y;z_i) \nabla y \right]\cdot \nabla z_\epsilon dx = \langle v, z_\epsilon \rangle, \quad i=1,2.
    \end{equation*}
    Subtracting these equations gives
    \begin{equation*}
        \int_\Omega a(y) |\nabla z_\epsilon|^2 + J_b(\nabla y) \nabla z_\epsilon \cdot \nabla z_\epsilon dx = \int_{\Omega} \left( a'(y;z_1) -a'(y; z_2) \right) \nabla y \cdot \nabla z_\epsilon dx.
    \end{equation*}
    Setting $M := \max\{ |y(x)| \mid x \in \overline \Omega\}$, we see from \cref{lem:Lip-Hadamard} that for a.e. $x \in \Omega$, the mapping $\eta\mapsto a'(y(x);\eta)$ is Lipschitz continuous with Lipschitz constant $C_{2M}$.
    From this, the non-negative definiteness of $J_b$, and \cref{ass:quasi_func1}, we have
    \begin{equation*}
        \begin{aligned}[t]
            a_0 \|\nabla z_\epsilon\|_{L^2(\Omega)}^2
            & \leq \int_{K_0 \setminus K_\epsilon} C_{2M} | z_1 - z_2| | \nabla y \cdot \nabla z_\epsilon | dx \\
            & \leq C_{2M} \epsilon \|\nabla y\|_{L^{2}(K_0 \setminus K_\epsilon)}\|\nabla z_\epsilon\|_{L^2(K_0 \setminus K_\epsilon)}.
        \end{aligned}
    \end{equation*}
    Proceeding as in the proof of \cref{thm:existence} yields $z_1 = z_2$.

    \medskip

    \emph{Step 2: Existence of solutions.} Assume that $y \in W^{1,\bar s}_0(\Omega)$ and $v \in H^{-1}(\Omega)$.
    We set $M:= \|y\|_{C(\overline{\Omega})}$. We fix $\eta \in L^{s_1}(\Omega)$ with $s_1$ as given in \eqref{eq:s1} and consider the equation
    \begin{equation} \label{eq:dire_deri2}
        \left\{
            \begin{aligned}
                -\div \left [\left(a(y)\Id + J_b(\nabla y) \right) \nabla z \right] & = v + \div\left[ a'(y; \eta) \nabla y \right] && \text{in } \Omega, \\
                z& =0 && \text{on } \partial\Omega.
            \end{aligned}
        \right.
    \end{equation}
    Setting $\tilde{a}(x) := a(y(x))\Id + J_b(\nabla y(x))$ and using \ref{ass:quasi_func1}, the non-negative definiteness of the matrix $J_b(\nabla y(x))$, as well as the global Lipschitz continuity of $b$, yields
    \begin{equation*}
        a_0 |\xi|^2 \leq \tilde{a}(x)\xi \cdot \xi \leq \left( C_MM + a(0) + L_b \right)|\xi|^2 \quad \text{for all } \xi \in \R^N \text{and a.e. }x\in \Omega.
    \end{equation*}
    Furthermore, since the mapping $\eta \mapsto a'(y(x);\eta)$ is Lipschitz continuous with Lipschitz constant $C_{2M}$ for a.e. $x \in \Omega$, we have that $|a'(y;\eta)| \leq C_{2M}|\eta|$ almost everywhere in $\Omega$ and hence that $a'(y;\eta) \in L^{s_1}(\Omega)$. From this and the choice of $s_1$ (see \eqref{eq:s1}), it follows that the right hand side of equation \eqref{eq:dire_deri2} belongs to $H^{-1}(\Omega)$. Equation \eqref{eq:dire_deri2} therefore admits a unique solution $z_\eta \in H^1_0(\Omega)$, which satisfies
    \begin{equation}
        \label{eq:apriori4}
        \begin{aligned}[t]
            \| \nabla z_\eta\|_{L^2(\Omega)} &\leq \frac{1}{a_0} \left(\| v \|_{H^{-1}(\Omega)} + \| a'(y; \eta) \nabla y \|_{L^2(\Omega)} \right) \\
            & \leq \frac{1}{a_0} \left(\| v \|_{H^{-1}(\Omega)} + C_{2M}\| |\eta| \nabla y \|_{L^2(\Omega)} \right) \\
            & \leq \frac{1}{a_0} \left(\| v \|_{H^{-1}(\Omega)} + C_{2M}|\Omega|^{1/s_2}\| \eta \|_{L^{s_1}(\Omega)}\| \nabla y \|_{L^{\bar s}(\Omega)} \right).
        \end{aligned}
    \end{equation}
    Here we have just used the H\"{o}lder inequality and the second relation in \eqref{eq:s1} to obtain the last estimate.
    Since the embedding $H^1_0(\Omega) \hookrightarrow L^{s_1}(\Omega)$ is continuous, we can define the operator $T: L^{s_1}(\Omega) \ni \eta \mapsto z_\eta \in L^{s_1}(\Omega)$.
    Let $\eta_1$ and $ \eta_2$ be arbitrary in $L^{s_1}(\Omega)$ and set $z_i := T(\eta_i)$, $i=1,2$. By simple calculation, we obtain
    \begin{equation*}
        \begin{aligned}
            \| \nabla z_1- \nabla z_2\|_{L^2(\Omega)} &\leq \frac{1}{a_0} \| \left(a'(y; \eta_1) - a'(y;\eta_2) \right) \nabla y \|_{L^2(\Omega)} \\
            & \leq \frac{1}{a_0} C_{2M}\| |\eta_1 - \eta_2| \nabla y \|_{L^2(\Omega)} \\
            & \leq \frac{1}{a_0} C_{2M}|\Omega|^{1/s_2}\| \eta_1 - \eta_2 \|_{L^{s_1}(\Omega)}\| \nabla y \|_{L^{\bar s}(\Omega)}.
        \end{aligned}
    \end{equation*}
    This implies the continuity of $T$. Furthermore, as a result of \eqref{eq:apriori4} and the compact embedding $H^1_0(\Omega) \Subset L^{s_1}(\Omega)$, $T$ is compact.
    We shall apply the Leray--Schauder principle to show that operator $T$ admits at least one fixed point $z$, which is then a solution to equation \eqref{eq:dire_deri}. To this end, we need prove the set
    \begin{equation*}
        K:= \left\{ \eta \in L^{s_1}(\Omega) \mid \exists t \in (0,1): \eta = T(t\eta) \right\}
    \end{equation*}
    is bounded.

    We now argue by contradiction. Assume that there exist sequences $\{\eta_k\} \subset K$ and $\{t_k\} \subset (0,1)$ such that $\eta_k = T(t_k\eta_k)$ and $\lim_{k \to \infty} \|\eta_k \|_{L^{s_1}(\Omega)} = \infty$.
    We have
    \begin{equation} \label{eq:deri1}
        \left\{
            \begin{aligned}
                -\div \left [\left(a(y)\Id + J_b(\nabla y) \right) \nabla \eta_k \right] &= v + \div\left[ a'(y; t_k\eta_k) \nabla y \right] && \text{in } \Omega, \\
                \eta_k& =0 && \text{on } \partial\Omega.
            \end{aligned}
        \right.
    \end{equation}
    Let us set $r_k := \frac{1}{\|\eta_k \|_{L^{s_1}(\Omega)}} \to 0$ and $\hat \eta_k := r_k\eta_k$.
    From this, \eqref{eq:deri1}, and the positive homogeneity of $a'(y;\cdot)$, we arrive at
    \begin{equation} \label{eq:deri2}
        \left\{
            \begin{aligned}
                -\div \left [\left(a(y)\Id + J_b(\nabla y) \right) \nabla \hat \eta_k \right]& = r_k v + \div\left[t_k a'(y; \hat\eta_k) \nabla y \right] && \text{in } \Omega, \\
                \hat \eta_k& =0 && \text{on } \partial\Omega.
            \end{aligned}
        \right.
    \end{equation}
    By simple computation, we deduce from $\| \hat \eta_k \|_{L^{s_1}(\Omega)}=1$ that
    \begin{equation}
        \label{eq:apriori5}
        \begin{aligned}[t]
            \| \nabla \hat \eta_k\|_{L^2(\Omega)} &\leq \frac{1}{a_0} \left(r_k\| v \|_{H^{-1}(\Omega)} + \|t_k a'(y;\hat \eta_k) \nabla y \|_{L^2(\Omega)} \right) \\
            & \leq \frac{1}{a_0} \left(r_k\| v \|_{H^{-1}(\Omega)} + C_{2M}|\Omega|^{1/s_2}\| \hat \eta_k \|_{L^{s_1}(\Omega)}\| \nabla y \|_{L^{\bar s}(\Omega)} \right)\\
            & = \frac{1}{a_0} \left(r_k\| v \|_{H^{-1}(\Omega)} + C_{2M}|\Omega|^{1/s_2}\| \nabla y \|_{L^{\bar s}(\Omega)} \right) \\
            & \leq C
        \end{aligned}
    \end{equation}
    for all $k \in \N$ and for some constant $C$ independent of $k \in \N$.
    From this and the compact embedding $H^1_0(\Omega) \Subset L^{s_1}(\Omega)$, we can assume that $\hat \eta_k \rightharpoonup \hat \eta$ in $H^1_0(\Omega)$ and $\hat \eta_k \to \hat \eta$ in $L^{s_1}(\Omega)$ for some $\hat \eta \in H^1_0(\Omega)$. The Lipschitz continuity of $a'(y(x);\cdot)$, for a.e. $x \in \Omega$, implies that $a'(y;\hat \eta_k) \to a'(y;\hat \eta)$ in $L^{s_1}(\Omega)$.
    Moreover, we can assume that $t_k \to t_0 \in [0,1]$. Letting $k \to \infty$ in equation \eqref{eq:deri2}, we see from the above arguments that
    \begin{equation} \label{eq:deri3}
        \left\{
            \begin{aligned}
                -\div \left [\left(a(y)\Id + J_b(\nabla y) \right) \nabla \hat \eta \right] & = \div\left[t_0 a'(y; \hat\eta) \nabla y \right] && \text{in } \Omega, \\
                \hat \eta &=0 && \text{on } \partial\Omega.
            \end{aligned}
        \right.
    \end{equation}
    The uniqueness of solutions to \eqref{eq:deri3} gives $\hat \eta = 0$, which is in contradiction to $\|\hat \eta\|_{L^{s_1}(\Omega)} = \lim_{k \to \infty}\|\hat \eta_k\|_{L^{s_1}(\Omega)} = 1$.
\end{proof}

We next show boundedness and continuity properties of \eqref{eq:dire_deri}.
\begin{theorem}\label{thm:dire_deri2}
    Let \cref{ass:quasi_func1,ass:quasi_func2} hold and let $\bar s$ be as given in \eqref{eq:U_set}.
    \begin{enumerate}[label=(\roman*)]
        \item
            If $\{y_n\}$ is bounded in $W^{1,\bar s}_0(\Omega)$ such that $\nabla y_n \to \nabla y$ in $L^2(\Omega)$ for some $y \in W^{1,\bar s}_0(\Omega)$, and if $\{v_n\}$ is bounded $H^{-1}(\Omega)$, then there exists a constant $C_2$ depending only on $a_0$, $\Omega$, $N$, $\bar s$, $\|y\|_{W^{1,\bar s}_0(\Omega)}$, $\sup\{\|y_n\|_{W^{1,\bar s}_0(\Omega)}\}$, and $\sup\{\|v_n\|_{H^{-1}(\Omega)}\}$ such that
            \begin{equation} \label{eq:apriori_DD}
                \|z(y_n, v_n)\|_{H^1_0{(\Omega)}} \leq C_2
            \end{equation}
            for all solutions $z(y_n,v_n)$ of \eqref{eq:dire_deri} corresponding to $y_n$ and $v_n$.

        \item
            If $v_n \to v$ in $H^{-1}(\Omega)$, then $z(y,v_n) \to z(y,v)$ in $H^1_0(\Omega)$.
    \end{enumerate}
\end{theorem}
\begin{proof}
    \emph{Ad (i):} Assume that $\{y_n\}$ is bounded in $W^{1,\bar s}_0(\Omega)$ such that $\nabla y_n \to \nabla y$ in $L^2(\Omega)$ for some $y \in W^{1,\bar s}_0(\Omega)$, and $\{v_n\}$ is bounded $H^{-1}(\Omega)$.
    Then there exists a constant $M>0$ such that
    \begin{equation}
        \| y\|_{C(\overline\Omega)}, \| y_n\|_{C(\overline\Omega)}, \| y\|_{W^{1,\bar s}_0(\Omega)} , \| y_n\|_{W^{1,\bar s}_0(\Omega)}\leq M
    \end{equation} for all $n \geq 1$.
    Let $z_n:= z(y_n, v_n)$ be the solution to \eqref{eq:dire_deri} corresponding to $y_n$ and $v_n$. We shall prove the boundedness of $\{z_n\}$ in $H^1_0(\Omega)$ by contradiction. Suppose that there exists a subsequence, again denoted by $\{z_n\}$, such that
    \begin{equation} \label{eq:contradiction}
        \lim_{n \to \infty} \| z_n \|_{H^1_0(\Omega)} = \infty.
    \end{equation}
    Since $z_n$ satisfies equation \eqref{eq:dire_deri} corresponding to $y:= y_n$ and $v := v_n$, one has
    \begin{equation*}
        \int_{\Omega} \left[a(y_n) + J_b(\nabla y_n) \right]\nabla z_n \cdot \nabla z_n dx = \left\langle v_n, z_n \right \rangle - \int_\Omega a'(y_n; z_n) \nabla y_n \cdot \nabla z_n dx.
    \end{equation*}
    Combining this with \cref{ass:quasi_func1}, the non-negative definiteness of the matrix $J_b(\nabla y_n(x))$, and the Lipschitz continuity of $a'(y_n(x);\cdot)$, the Hölder inequality and the second relation in \eqref{eq:s1} lead to
    \begin{equation} \label{eq:zn_esti}
        \begin{aligned}[t]
            \|\nabla z_n\|_{L^2(\Omega)} &\leq \frac{1}{a_0} \left(\| v_n \|_{H^{-1}(\Omega)} + \|a'(y_n; z_n) \nabla y_n \|_{L^2(\Omega)} \right) \\
            & \leq \frac{1}{a_0} \left(\| v_n \|_{H^{-1}(\Omega)} + C_{2M}|\Omega|^{1/s_2}\| z_n \|_{L^{s_1}(\Omega)}\| \nabla y_n \|_{L^{\bar s}(\Omega)} \right).
        \end{aligned}
    \end{equation}
    From this and \eqref{eq:contradiction}, we obtain $\lim_{n \to \infty} \| z_n \|_{L^{s_1}(\Omega)} = \infty$.
    Setting $t_n:= \frac{1}{\|z_n \|_{L^{s_1}(\Omega)}}$ and $\hat z_n := t_n z_n$
    yields that
    \begin{equation} \label{eq:lim}
        t_n \to 0 \quad \text{and} \quad \lim_{n \to \infty} \| \hat z_n \|_{L^{s_1}(\Omega)} = 1.
    \end{equation}
    On the other hand, $\hat z_n$ satisfies
    \begin{equation} \label{eq:deri4}
        \left\{
            \begin{aligned}
                -\div \left [\left(a(y_n)\Id + J_b(\nabla y_n) \right) \nabla \hat z_n \right] & = t_n v_n + \div\left[a'(y_n; \hat z_n) \nabla y_n \right] && \text{in } \Omega, \\
                \hat z_n &= 0 && \text{on } \partial\Omega.
            \end{aligned}
        \right.
    \end{equation}
    The same argument as in \eqref{eq:zn_esti} gives
    \begin{equation*}
        \|\nabla \hat z_n\|_{L^2(\Omega)} \leq \frac{1}{a_0} \left(t_n\| v_n \|_{H^{-1}(\Omega)} + C_{2M} |\Omega|^{1/s_2}\|\hat z_n \|_{L^{s_1}(\Omega)}\| \nabla y_n \|_{L^{\bar s}(\Omega)} \right)
        \leq C
    \end{equation*}
    for all $n \in \N$ and for some constant $C$ independent of $n \in \N$.
    Consequently, $\{\hat z_n\}$ is bounded in $H^1_0(\Omega)$. We can thus extract a subsequence, denoted in the same way, such that $\hat z_n \rightharpoonup \hat z$ in $H^1_0(\Omega)$ and $\hat z_n \to \hat z$ in $L^{s_1}(\Omega)$ for some $\hat z \in H^1_0(\Omega)$. We write $a'(y_n; \hat z_n) = c_n \hat z_n$, where
    \begin{equation}
        c_n(x) :=
        \begin{cases}
            \frac{a'(y_n(x); \hat z_n(x))}{\hat z_n(x)} & \text{if } \hat z_n(x) \neq 0,\\
            0 & \text{otherwise}
        \end{cases}
    \end{equation} for a.e. $x \in \Omega$.
    We have $|c_n(x)| \leq C_{2M}$ because of the Lipschitz continuity of $a'(y_n(x); \cdot)$ for a.e. $x \in \Omega$. Again, by using a subsequence, we can assume that
    $c_n \hat z_n \rightharpoonup c\hat z$ in $L^{s_0}(\Omega)$ for any $s_0 \in [1,s_1)$, particularly for $s_0$ with $s_0^{-1}+ \bar s^{-1} = 2^{-1}$.
    Since $\{y_n\}$ is bounded in $W^{1,\bar s}_0(\Omega)$, we can assume that $y_n \to y$ in $C(\overline{\Omega})$, as a result of the compact embedding $W^{1,\bar s}(\Omega) \Subset C(\overline{\Omega})$. Consequently, $a(y_n) \to a(y)$ in $C(\overline\Omega)$. Moreover, the Lebesgue dominated convergence theorem together with the fact that $\nabla y_n \to \nabla y$ in measure implies that $J_b(\nabla y_n) \to J_b(\nabla y)$ in $L^{m}(\Omega)$ for all $m \geq 1$. Letting $n \to \infty$ in equation \eqref{eq:deri4}, we arrive at
    \begin{equation} \label{eq:deri5}
        \left\{
            \begin{aligned}
                -\div \left [\left(a(y)\Id + J_b(\nabla y) \right) \nabla \hat z \right] &= \div\left[c \hat z \nabla y \right] && \text{in } \Omega, \\
                \hat z &=0 && \text{on } \partial\Omega.
            \end{aligned}
        \right.
    \end{equation}
    As in the proof of \cref{thm:dire_deri1}, we can show that \eqref{eq:deri5} has at most one solution and hence that $\hat z = 0$. However, by virtue of the second limit in \eqref{eq:lim}, we have $\|\hat z \|_{L^{s_1}(\Omega)} = 1$, which yields a contradiction.

    We have thus shown that sequence $\{z_n\}$ is bounded in $H^1_0(\Omega)$. From estimate \eqref{eq:zn_esti} and the choice of $s_1, s_2$, the upper bound of $\{\|z_n\|_{H^1_0(\Omega)}\}$ depends only on $M, \sup\{ \|v_n\|_{H^{-1}(\Omega)} \}, a_0, \Omega$, and $\bar s, s_1, s_2$ and so on $N$. This shows the a priori estimate \eqref{eq:apriori_DD}.

    \medskip

    \emph{Ad (ii):} Assume now that $v_n \to v$ in $H^{-1}(\Omega)$. Setting $\tilde{z}_n := z(y, v_n)$, we see from \eqref{eq:apriori_DD} and the compact embedding $H^1_0(\Omega) \Subset L^{s_1}(\Omega)$ that
    \begin{equation} \label{eq:lim_dd}
        \tilde{z}_{n_k} \rightharpoonup \tilde{z } \text{ in } H^1_0(\Omega) \quad \text{and} \quad \tilde{z}_{n_k} \to \tilde{z } \text{ in } L^{s_1}(\Omega)
    \end{equation}
    for some subsequence $\{n_k\} \subset \N$ and some function $\tilde{z} \in H^1_0(\Omega)$. By letting $k \to \infty$, the uniqueness of solutions to \eqref{eq:dire_deri} guarantees $\tilde{z} = z(y,v)$. On the other hand, $\tilde{z}_{n_k}$ and $\tilde{z}$ satisfy the equation
    \begin{equation*}
        \left \{
            \begin{aligned}
                -\div \left [\left(a(y)\Id + J_b(\nabla y) \right) \nabla (\tilde{z}_{n_k} - \tilde{z}) + \left( a'(y; \tilde{z}_{n_k}) - a'(y; \tilde{z}) \right) \nabla y \right] &= v_{n_k} - v && \text{in } \Omega, \\
                \tilde{z}_{n_k} - \tilde{z} &=0 && \text{on } \partial\Omega.
            \end{aligned}
        \right.
    \end{equation*}
    The same arguments as above yield that
    \begin{equation*}
        \begin{aligned}[t]
            \|\nabla ( \tilde{z}_{n_k} - \tilde z) \|_{L^2(\Omega)} &\leq \frac{1}{a_0} \left(\| v_{n_k} - v \|_{H^{-1}(\Omega)} + \|(a'(y; \tilde{z}_{n_k}) - a'(y; \tilde{z})) \nabla y \|_{L^2(\Omega)} \right) \\
            & \leq \frac{1}{a_0} \left(\| v_{n_k} - {v}\|_{H^{-1}(\Omega)} + C_{2M}|\Omega|^{1/s_2}\| \tilde{z}_{n_k} - \tilde z\|_{L^{s_1}(\Omega)}\| \nabla y\|_{L^{\bar s}(\Omega)} \right),
        \end{aligned}
    \end{equation*}
    which together with the second limit in \eqref{eq:lim_dd} gives $\tilde{z}_{n_k} \to \tilde{z}$ in $H^1_0(\Omega)$ as $k \to \infty$. 
    Recall that $s_1$ and $s_2$ are defined by \eqref{eq:s1}.
    From this and the uniqueness of solutions to \eqref{eq:dire_deri}, we obtain $\tilde{z}_{n} \to \tilde{z}$ in $H^1_0(\Omega)$ as $n \to \infty$.

\end{proof}

As a result of the compact embedding $L^p(\Omega) \Subset W^{-1, \tilde p}(\Omega)$ with $\tilde p$ as given in \eqref{eq:p_star}, \cref{thm:dire_deri2,thm:regu}, we have the following corollary.
\begin{corollary} \label{cor:dire_deri}
    Let \cref{ass:quasi_func1,ass:quasi_func2} hold true. Assume that $U$ is the open ball given by \eqref{eq:U_set} and that $V$ is a bounded set in $H^{-1}(\Omega)$. For each $u \in U$ and $v \in V$, let $z_{u,v}$ stand for the solution to \eqref{eq:dire_deri} corresponding to $y := y_u$ and $v$.
    Then, there exists a constant $C_3 := C_3\left(a_0, \Omega, N, p, U, V\right )$ such that
    \begin{equation}
        \| z_{u,v}\|_{H^1_0(\Omega) } \leq C_3 \quad \text{for all} \quad u \in U, v \in V.
    \end{equation}
\end{corollary}

\begin{theorem} \label{thm:DD}
    Assume that \cref{ass:quasi_func1,ass:quasi_func2} are valid and that $U$ is the open ball in $L^p(\Omega)$ defined as in \eqref{eq:U_set}. Then $S: U \to H^1_0(\Omega)$ is directional differentiable. Moreover, for any $u \in U$ and $v \in L^p(\Omega)$, $z := S'(u;v)$ is the unique solution in $H^1_0(\Omega)$ of the equation
    \begin{equation*}
        \left\{
            \begin{aligned}
                -\div \left [\left(a(y_u)\Id + J_b(\nabla y_u) \right) \nabla z + a'(y_u; z) \nabla y_u \right] & = v && \text{in } \Omega, \\
                z &=0 && \text{on } \partial\Omega.
            \end{aligned}
        \right.
    \end{equation*}
\end{theorem}
\begin{proof}
    For any $u \in U$ and $v \in L^p(\Omega)$, we set $y := S(u)$, $y_\rho := S(u + \rho v)$, and $z_\rho := \frac{y_\rho - y}{\rho}$ for $\rho > 0$. A simple computation shows that
    \begin{equation} \label{eq:dd_S}
        \left\{
            \begin{aligned}
                -\div\left[ \frac{a(y + \rho z_\rho)-a(y)}{\rho} \nabla y_\rho + a(y) \nabla z_\rho + \frac{b(\nabla y + \rho \nabla z_\rho) - b(\nabla y)}{\rho} \right] &= v && \text{in } \Omega, \\
                z_\rho & = 0 && \text{on } \partial\Omega.
            \end{aligned}
        \right.
    \end{equation}
    Multiplying the above equation by $z_\rho$, integrating over $\Omega$, and using integration by parts, we see from \cref{ass:quasi_func1} and the monotonicity of $b$ that
    \begin{equation*}
        a_0 \| \nabla z_\rho \|_{L^2(\Omega)} \leq \| v \|_{H^{-1}(\Omega)} +\left \| \frac{a(y + \rho z_\rho)-a(y)}{\rho} \nabla y_\rho \right \|_{L^2(\Omega)}.
    \end{equation*}
    This together with \cref{ass:quasi_func1} gives
    \begin{equation*}
        \begin{aligned}
            a_0 \| \nabla z_\rho \|_{L^2(\Omega)} \leq \| v \|_{H^{-1}(\Omega)} + C_M \left \| |z_\rho| \nabla y_\rho \right \|_{L^2(\Omega)}
        \end{aligned}
    \end{equation*}
    with $M:= \sup\{ \|y\|_{C(\overline{\Omega})}, \|y_\rho\|_{C(\overline{\Omega})}: \rho \in (0,\hat \rho) \}$ for $\hat \rho$ small enough. Hölder's inequality and the embedding $L^p(\Omega) \hookrightarrow H^{-1}(\Omega)$ yield that
    \begin{equation}
        \label{eq:bound2}
        a_0 \| \nabla z_\rho \|_{L^2(\Omega)} \leq C(\Omega, p)\| v \|_{L^p(\Omega)} + C_M |\Omega|^{1/s_2} \left \| z_\rho \right \|_{L^{s_1}(\Omega)} \left \| \nabla y_\rho \right \|_{L^{\bar s}(\Omega)}
    \end{equation}
    with $s_1, s_2$ defined as in \eqref{eq:s1} and some constant $C(\Omega, p)$.

    We now show the boundedness of $\{z_\rho\}$ in $H^1_0(\Omega)$ by an indirect proof that is based on arguments similar to the ones in the proof of estimate \eqref{eq:apriori_DD}. Assume that $\{z_\rho\}$ is not bounded in $H^1_0(\Omega)$. A subsequence $\{\rho_k\}$ then exists such that $\rho_k \to 0^+$ and $\| \nabla z_{\rho_k} \|_{L^2(\Omega)} \to \infty$. By virtue of \eqref{eq:bound2}, it thus holds that $ \| z_{\rho_k} \|_{L^{s_1}(\Omega)} \to \infty$.

    Again, setting $ t_k:= \frac{1}{\|z_{\rho_k} \|_{L^{s_1}(\Omega)}}$, $\sigma_k := \frac{\rho_k}{t_k}$, and $\hat z_k := t_k z_{\rho_k}$
    yields that
    \begin{equation*}
        t_k \to 0, \quad y_{\rho_k} = y + \sigma_k \hat z_k, \quad \text{and} \quad \| \hat z_k \|_{L^{s_1}(\Omega)} = 1.
    \end{equation*}
    Due to \cref{thm:regu}, $y_{\rho_k} \to y$ in $ C(\overline{\Omega})$ and so in $L^{s_1}(\Omega)$. It therefore holds that $\sigma_k \to 0^+$.
    On the other hand, $\hat z_k$ satisfies
    \begin{equation}
        \label{eq:dd_sigma}
        \left\{
            \begin{aligned}
                -\div\left[ \frac{a(y + \sigma_k \hat z_k)-a(y)}{\sigma_k} \nabla y_{\rho_k} + a(y)\nabla \hat z_k + \frac{b(\nabla y + \sigma_k \nabla \hat z_k) - b(\nabla y)}{\sigma_k} \right] & = t_k v && \text{in } \Omega\\
                \hat z_k &= 0 && \text{on } \partial\Omega.
            \end{aligned}
        \right.
    \end{equation}
    From this, we obtain the boundedness of $\{\hat z_k\}$ in $H^1_0(\Omega)$. We can then extract a subsequence, again denoted by $\{\hat z_k\}$, such that $\hat z_k \rightharpoonup \hat z$ in $H^1_0(\Omega)$, $\hat z_k \to \hat z$ in $L^{s_1}(\Omega)$, $\hat z_k(x) \to \hat z(x)$, and $|\hat z_k(x)| \leq g(x)$ for all $k \in \N$, for a.e. $x\in \Omega$, and for some $g \in L^{s_1}(\Omega)$. Since $a:\R\to\R$ satisfies \cref{ass:quasi_func1}, we have as a result of \cref{lem:Lip-Hadamard} that
    \begin{equation*}
        \frac{a(y(x) + \sigma_k \hat z_k(x))-a(y(x))}{\sigma_k} \to a'(y(x); \hat z(x))\quad \text{for a.e. }x \in \Omega.
    \end{equation*}
    Furthermore, \cref{ass:quasi_func1} also gives
    \begin{equation*}
        \left| \frac{a(y(x) + \sigma_k \hat z_k(x)-a(y(x))}{\sigma_k} \right| \leq C_M |\hat z_k(x)| \leq C_M g(x)\quad \text{for a.e. }x \in \Omega.
    \end{equation*}
    The Lebesgue dominated convergence theorem thus implies that
    \begin{equation} \label{eq:limit1}
        \frac{a(y + \sigma_k \hat z_k)-a(y)}{\sigma_k} \to a'(y; \hat z) \quad \text{in } L^{s_1}(\Omega).
    \end{equation}
    Rewriting
    \begin{equation*}
        \begin{aligned}
            \frac{b(\nabla y + \sigma_k \nabla \hat z_k) - b(\nabla y)}{\sigma_k} & = \frac{b(\nabla y_{\rho_k}) - b(\nabla y)}{\sigma_k} \\
            & = J_b(\nabla y + \theta_k (\nabla y_{\rho_k} - \nabla y) )\nabla \hat z_k, \quad (\theta_k := \theta_k(x) \in (0,1)),
        \end{aligned}
    \end{equation*}
    using the fact that $y_{\rho_k} \to y$ in $W^{1,\bar s}_0(\Omega)$ and the boundedness of $J_b$, we now apply the Lebesgue dominated convergence theorem for a subsequence, which is denoted in the same way, to obtain
    \begin{equation} \label{eq:limit2}
        \frac{b(\nabla y + \sigma_k \nabla \hat z_k) - b(\nabla y)}{\sigma_k} \rightharpoonup J_b(\nabla y) \nabla\hat z \quad \text{in } L^{s_3}(\Omega)
    \end{equation} for any $s_3 \in [1,2)$. Letting $k \to \infty$ in equation \eqref{eq:dd_sigma} and using limits \eqref{eq:limit1} and \eqref{eq:limit2} yields that
    \begin{equation*}
        \left \{
            \begin{aligned}
                -\div \left [\left(a(y)\Id + J_b(\nabla y) \right) \nabla \hat z + a'(y; \hat z) \nabla y \right] &= 0 && \text{in } \Omega, \\
                \hat z &=0 && \text{on } \partial\Omega.
            \end{aligned}
        \right.
    \end{equation*}
    The uniqueness of solutions implies that $\hat z = 0$, a contradiction of the fact that $\| \hat z \|_{L^{s_1}(\Omega)} = \lim_{k \to \infty}\| \hat z_k \|_{L^{s_1}(\Omega)} =1$.

    Having proved the boundedness of $\{z_\rho\}$ in $H^1_0(\Omega)$, we can assume that $z_\rho \rightharpoonup z$ in $H^1_0(\Omega)$ and $z_\rho \to z$ in $L^{s_1}(\Omega)$. From this and standard arguments as above, we obtain the desired conclusion.
\end{proof}

\subsection{Regularization of the control-to-state operator}

To derive C-stationarity conditions, we apply the adapted penalization method of Barbu \cite{Barbu1984}.
We consider a regularization of the state equation via a classical mollification of the non-smooth nonlinearity.
Let $\psi$ be a non-negative function in $C^\infty_0(\R)$ such that $\mathrm{supp}(\psi) \subset [-1,1]$, $\int_{\R} \psi(\tau)d\tau = 1$
and define the family $\{a_\epsilon\}_{\epsilon>0}$ of functions
\begin{equation}
    \label{eq:a_epsilon}
    a_\epsilon := \frac{1}{\epsilon}a * \left(\psi\circ (\epsilon^{-1}{\Id})\right),
\end{equation}
where $f *g$ stands for the convolution of $f$ and $g$.
Then, $a_\epsilon \in C^\infty(\R)$ by a standard result.
In addition, a simple calculation shows that
\begin{equation} \label{eq:ass1}
    a_\epsilon(\tau) \geq a_0 \quad \forall \tau \in \R.
\end{equation}
Moreover, for any $M>0$,
\begin{align} \label{eq:a_regu}
    |a_\epsilon(\tau) - a(\tau) | &\leq C_{M+1}\epsilon && \text{for all} \quad \tau \in \R, \epsilon \in (0,1)
    \shortintertext{and}
    |a_\epsilon (\tau_1) - a_\epsilon(\tau_2) | &\leq C_{M+1} |\tau_1 - \tau_2 | && \text{for all} \quad \tau_i \in \R, |\tau_i | \leq M, i=1,2, \epsilon \in (0,1)\label{eq:a_regu_Lip}
\end{align}
with $C_{M+1}$ given in \cref{ass:quasi_func1}. We now consider the regularized equation
\begin{equation} \label{eq:RE_state}
    \left\{
        \begin{aligned}
            -\div [a_\epsilon(y)\nabla y + b(\nabla y)] &= u && \text{in } \Omega, \\
            y &=0 && \text{on } \partial\Omega.
        \end{aligned}
    \right.
\end{equation}
Since $a_\epsilon$ satisfies \cref{ass:quasi_func1}, \cref{thm:existence} yields that equation \eqref{eq:RE_state} admits for each $u \in L^p(\Omega)$ with $p > N/2$ a unique solution $y \in H^1_0(\Omega) \cap C(\overline{\Omega})$.

In the sequel, for each $\epsilon>0$, we denote by $S_\epsilon:L^p(\Omega)\to H^1_0(\Omega)$ the solution operator of \eqref{eq:RE_state}, which by \cref{thm:existence} satisfies the a priori estimate
\begin{equation} \label{eq:RE_apriori}
    \|S_\epsilon(u)\|_{H^{1}_0(\Omega)} + \| S_\epsilon(u)\|_{C(\overline{\Omega})} \leq C_\infty \|u\|_{L^p(\Omega)} \quad \text {for all } u \in L^p(\Omega)
\end{equation}
with the constant $C_\infty$ defined as in \cref{thm:existence} corresponding to $p^* := \tilde p$, where $\tilde{p}$ is defined as in \eqref{eq:p_star}. The following regularity of solutions is a direct consequence of \cref{thm:regu}.
\begin{corollary} \label{cor:RE_regu}
    Assume that \cref{ass:quasi_func1,ass:quasi_func2} hold true. Let $U$ and $\bar s$ be defined as in \eqref{eq:U_set}. Then, the following assertions are valid.
    \begin{itemize}
        \item[(i)] If $u \in U$, then $S_\epsilon(u) \in W^{1,\bar s}_0(\Omega)$ and
            \begin{equation} \label{eq:RE_regu}
                \|S_\epsilon(u)\|_{W^{1,\bar s}_0{(\Omega)}} \leq C_4
            \end{equation}
            for some constant $C_4$ depending only on $a_0, \Omega, N, p$, and $U$.
        \item[(ii)] If $u_n \rightharpoonup u$ in $L^p(\Omega)$ with $u_n, u \in U$, then $S_\epsilon(u_n) \to S_\epsilon(u)$ in $W^{1,\bar s}_0(\Omega) \cap C(\overline{\Omega})$.
    \end{itemize}
\end{corollary}

We now study the differentiability of $S_\epsilon$. For $y \in W^{1,\bar s}_0{(\Omega)}$, we consider the equation
\begin{equation} \label{eq:RE_dire_deri}
    \left\{
        \begin{aligned}
            -\div \left [\left(a_\epsilon(y)\Id + J_b(\nabla y) \right) \nabla z + a'_\epsilon(y) z \nabla y \right] & = v && \text{in } \Omega, \\
            z &= 0 && \text{on } \partial\Omega.
        \end{aligned}
    \right.
\end{equation}
The well-posedness of equation \eqref{eq:RE_dire_deri} is proven analogously to \cref{thm:dire_deri1,thm:dire_deri2}.
\begin{theorem} \label{thm:RE_dire_deri}
    Let \cref{ass:quasi_func1,ass:quasi_func2} hold and let $\bar s$ be defined as in \eqref{eq:U_set}. Assume that $y \in W^{1,\bar s}_0(\Omega)$. Then, for each $v \in H^{-1}(\Omega)$, equation \eqref{eq:RE_dire_deri} admits a unique solution $z \in H^1_0(\Omega)$.

    Moreover, if $\{y_n\}$ is bounded in $W^{1,\bar s}_0(\Omega)$ such that $\nabla y_n \to \nabla y$ in $L^2(\Omega)$ for some $y \in W^{1,\bar s}_0(\Omega)$, and $\{v_n\}$ is bounded $H^{-1}(\Omega)$, then there exists a constant $C_5$ independent of $\epsilon$ and $n\in\N$ such that
    \begin{equation} \label{eq:RE_apriori_DD}
        \|z_n\|_{H^1_0{(\Omega)}} \leq C_5
    \end{equation}
    for all solutions $z_n$ of \eqref{eq:RE_dire_deri} corresponding to $y_n$ and $v_n$.
\end{theorem}

Also, similarly to \cref{thm:DD,cor:dire_deri}, we obtain the G\^{a}teaux differentiability of $S_\epsilon$.
\begin{theorem}
    Let \cref{ass:quasi_func1,ass:quasi_func2} hold. Assume that $U$ is given as in \eqref{eq:U_set}. Then, $S_\epsilon: U \to H^1_0(\Omega)$ is G\^{a}teaux differentiable. Moreover, for any $u \in U$ and $ v \in L^p(\Omega)$, the Gâteaux derivative $z := S'_\epsilon(u)v$ is the unique solution in $H^1_0(\Omega)$ of the equation
    \begin{equation} \label{eq:T_oper}
        \left\{
            \begin{aligned}
                -\div \left [\left(a_\epsilon(y)\Id + J_b(\nabla y) \right) \nabla z + a'_\epsilon(y) z \nabla y \right] &= v && \text{in } \Omega, \\
                z &=0 && \text{on } \partial\Omega,
            \end{aligned}
        \right.
    \end{equation}
    where $y := S_\epsilon(u)$.
\end{theorem}

For later use in \cref{sec:reg} (see \cref{thm:OS}) we also need the uniform boundedness of solutions to \eqref{eq:T_oper}. For this purpose, we define the operator $T_{y, \epsilon}: H^1_0(\Omega) \to H^{-1}(\Omega)$ by
\begin{equation*}
    T_{y,\epsilon} w := -\div \left [\left(a_\epsilon(y)\Id + J_b(\nabla y) \right) \nabla w + a'_\epsilon(y) w \nabla y \right].
\end{equation*}
\begin{proposition} \label{prop:T_iso}
    Let $\bar s$ be defined as in \eqref{eq:U_set}.
    For any $y \in W^{1,\bar s}_0(\Omega)$ and $\epsilon >0$, the operator $T_{y, \epsilon}: H^1_0(\Omega) \to H^{-1}(\Omega)$ is an isomorphism.
    Moreover, if $\{y_\epsilon\}$ is bounded in $W^{1,\bar s}_0(\Omega)$ such that $\nabla y_\epsilon \to \nabla y$ in $L^2(\Omega)$ for some $y \in W^{1,\bar s}_0(\Omega)$, then
    \begin{equation} \label{eq:T_esti}
        \sup\left\{ \| T^{-1}_{y_\epsilon, \epsilon} \|_{\mathcal L(H^{-1}(\Omega), H^1_0(\Omega))} \mid \epsilon \in (0,1) \right \} < \infty.
    \end{equation}
\end{proposition}
\begin{proof}
    From the definition, we immediately deduce that $T_{y, \epsilon}$ is continuous. Moreover, \cref{thm:RE_dire_deri} yields that $T_{y, \epsilon}$ is bijective, while the estimate \eqref{eq:T_esti} follows from \eqref{eq:RE_apriori_DD}.

    It remains to prove that $T_{y,\epsilon}^{-1}$ is continuous.
    Let $v_n \to v$ in $ H^{-1}(\Omega)$ and set $z_n := T_{y,\epsilon}^{-1}v_n$, $z := T_{y,\epsilon}^{-1}v$. It is easy to see that $z_n$ and $ z$ satisfy the equation
    \begin{equation} \label{eq:T_con}
        \left \{
            \begin{aligned}
                -\div \left [\left(a_\epsilon(y)\Id + J_b(\nabla y) \right) \nabla (z_n-z) + a'_\epsilon(y) (z_n-z) \nabla y \right] &= v_n-v && \text{in } \Omega, \\
                z_n-z &=0 && \text{on } \partial\Omega,
            \end{aligned}
        \right.
    \end{equation}
    which, together with \eqref{eq:a_regu_Lip}, implies that
    \begin{equation} \label{eq:T_esti1}
        \|\nabla (z_n - z)\|_{L^2(\Omega)} \leq \frac{1}{a_0} \left(\| v_n - v \|_{H^{-1}(\Omega)} + C_{M+1} |\Omega|^{1/s_2}\|z_n - z \|_{L^{s_1}(\Omega)}\| \nabla y\|_{L^{\bar s}(\Omega)} \right),
    \end{equation}
    where $M:= \| y\|_{C(\overline{\Omega})}$ and $s_1, s_2$ are defined as in \eqref{eq:s1}. The same argument as in \cref{thm:dire_deri2} then implies that $\|\nabla (z_n - z)\|_{L^2(\Omega)}$ is bounded. We can thus extract a subsequence, also denoted by $\{z_n - z\} $, such that $z_n - z \rightharpoonup \tilde{z}$ in $H^1_0(\Omega)$ and $z_n - z \to \tilde{z}$ in $L^{s_1}(\Omega)$  for some $\tilde{z} \in H^1_0(\Omega)$. Letting $n \to \infty$ in \eqref{eq:T_con} yields
    \begin{equation*}
        \left\{
            \begin{aligned}
                -\div \left [\left(a_\epsilon(y)\Id + J_b(\nabla y) \right) \nabla \tilde{z} + a'_\epsilon(y) \tilde{z} \nabla y \right] & = 0 && \text{in } \Omega, \\
                \tilde{z} &=0 && \text{on } \partial\Omega,
            \end{aligned}
        \right.
    \end{equation*}
    which together with the uniqueness of solutions indicates that $\tilde{z} = 0$. By virtue of this and the limit $z_n - z \to \tilde{z}$ in $L^{s_1}(\Omega)$, the estimate \eqref{eq:T_esti1} shows that $z_n \to z$ in $H^1_0(\Omega)$. Consequently, $T_{y,\epsilon}^{-1}$ is continuous.
\end{proof}

Since $T_{y, \epsilon}$ is isomorphic, so is its adjoint $T_{y,\epsilon}^*$, which immediately yields well-posedness of the regularized adjoint equation
\begin{equation} \label{eq:adjoint}
    \left\{
        \begin{aligned}
            -\div \left [\left(a_\epsilon(y)\Id + J_b(\nabla y)^T \right)\nabla w \right] + a'_\epsilon(y) \nabla y \cdot \nabla w &= v && \text{in } \Omega, \\
            w &=0 && \text{on } \partial\Omega,
        \end{aligned}
    \right.
\end{equation}
where $A^T$ stands for the transpose of matrix $A$.
\begin{corollary} \label{cor:adjoint_T}
    Let $\bar s$ be defined as in \eqref{eq:U_set}.
    Under \cref{ass:quasi_func1,ass:quasi_func2}, for any $y \in W^{1,\bar s}_0(\Omega)$, $v \in H^{-1}(\Omega)$, and $\epsilon > 0$, the equation \eqref{eq:adjoint}
    admits a unique solution $w \in H^1_0(\Omega)$.
\end{corollary}

Finally, we address convergence of $S_\epsilon$ to $S$ as $\epsilon\to 0$.
\begin{proposition} \label{lem:cont}
    If $u_\epsilon \rightharpoonup u$ in $L^p(\Omega)$ with $u_\epsilon, u \in U$, then $S_\epsilon(u_\epsilon) \to S(u)$ in $H^1_0(\Omega) \cap C(\overline{\Omega})$.
\end{proposition}
\begin{proof}
    Set $y_\epsilon := S_\epsilon(u_\epsilon)$.
    Then by \cref{cor:RE_regu}, a constant $M>0$ exists such that
    \begin{equation}
        \|y_\epsilon\|_{C(\overline{\Omega})} + \|y_\epsilon\|_{W^{1,\bar s}_0({\Omega})} \leq M \quad \text{for all} \quad \epsilon > 0
    \end{equation}
    with $\bar s$ given in \eqref{eq:U_set}.
    On the other hand, we have
    \begin{equation*}
        \left\{
            \begin{aligned}
                -\div \left [a_\epsilon(y_\epsilon)\nabla y_\epsilon + b(\nabla y_\epsilon) \right] &= u_\epsilon && \text{in } \Omega, \\
                y_\epsilon &=0 && \text{on } \partial\Omega.
            \end{aligned}
        \right.
    \end{equation*}
    Rewriting this equation as
    \begin{equation*}
        \left\{
            \begin{aligned}
                -\div \left [a(y_\epsilon)\nabla y_\epsilon + b(\nabla y_\epsilon) \right] &= u_\epsilon + v_\epsilon && \text{in } \Omega, \\
                y_\epsilon &=0 && \text{on } \partial\Omega
            \end{aligned}
        \right.
    \end{equation*}
    with $v_\epsilon := \div \left [\left(a_\epsilon(y_\epsilon) - a(y_\epsilon) \right) \nabla y_\epsilon \right] $, we then have $y_\epsilon = S(u_\epsilon + v_\epsilon)$. In addition, we deduce from \eqref{eq:a_regu} that
    \begin{equation}
        \begin{aligned}[t]
            \| v_\epsilon\|_{W^{-1,\bar s}(\Omega)} & = \sup\left\{ \int_{\Omega} \left(a_\epsilon(y_\epsilon) - a(y_\epsilon) \right) \nabla y_\epsilon \cdot \nabla \varphi dx \mid \| \varphi \|_{W_0^{1,\bar s'}(\Omega)} \leq 1, \bar s' = \frac{\bar s}{\bar s -1} \right \} \\
            & \leq C_{M+1} \epsilon \| \nabla y_\epsilon \|_{L^{\bar s}(\Omega)} \\
            & \leq C_{M+1} M \epsilon.
        \end{aligned}
    \end{equation}
    It follows that $v_\epsilon \to 0$ in $W^{-1,\bar s}(\Omega)$ as $\epsilon \to 0^+$. Since $L^p(\Omega) \Subset W^{-1,\tilde{p}}(\Omega)$ is compact, we have $u_\epsilon \to u$ in $W^{-1,\tilde{p}}(\Omega)$ and so $u_\epsilon + v_\epsilon \to u$ in $W^{-1,s_0}(\Omega)$ with $s_0 := \min\{\tilde{p},\bar s\} >N$. \cref{cor:continuity} then implies that $y_\epsilon = S(u_\epsilon + v_\epsilon) \to S(u) =y$ in $H^1_0(\Omega) \cap C(\overline{\Omega})$, as claimed.
\end{proof}

\section{Existence and optimality conditions} \label{sec:OS}

We now turn to the optimal control problem \eqref{eq:P}, which we recall is given as
\begin{equation}\label{eq:P2}
    \tag{P}
    \left\{
        \begin{aligned}
            \min_{u\in L^p(\Omega), y\in H^1_0(\Omega)} &J(y,u) \\
            \text{s.t.} \quad
            &\begin{aligned}[t]
                -\div [a(y)\nabla y + b(\nabla y)] &= u &&\text{in } \Omega\\
                y &= 0 &&\text{on } \partial\Omega.
            \end{aligned}
        \end{aligned}
    \right.
\end{equation}
It will frequently be useful to rewrite problem \eqref{eq:P2} using the control-to-state operator $S$ in the reduced form
\begin{equation} \label{eq:P21}
    \min_{u\in L^p(\Omega)} j(u) := J(S(u),u).
\end{equation}

We first address the existence of minimizers.
\begin{proposition}
    \label{prop:existence}
    Under \crefrange{ass:quasi_func1}{ass:cost_func}, there exists a minimizer $(\bar y, \bar u) \in H^1_0(\Omega) \times L^p(\Omega)$ of \eqref{eq:P2}.
\end{proposition}
\begin{proof}
    Applying \cref{thm:existence} to the case where $p^* := \tilde{p}$ with $\tilde{p}$ defined as in \eqref{eq:p_star} yields that
    \begin{equation*}
        \| S(u) \|_{H^1_0(\Omega)} + \| S(u) \|_{C(\overline\Omega)} \leq M\| u \|_{L^p(\Omega)}
    \end{equation*}
    for all $u \in L^p(\Omega)$ and for some constant $M$ independent of $u$. By \cref{ass:cost_func}, there exists a function $g_M: [0, \infty) \to \R$ such that
    $\lim_{t \to \infty}g_M(t) = +\infty$ and $j(u) \geq g_M(\|u \|_{L^p(\Omega)})$ for all $u \in L^p(\Omega)$. This implies that $j$ is coercive. Together with the weak lower semicontinuity of $j$, the existence of a minimizer follows by Tonelli's direct method.
\end{proof}

The remainder of this section is devoted to deriving optimality conditions for \eqref{eq:P2}. The weakest conditions are the primal stationarity conditions, which are also obtained by standard arguments.
\begin{proposition}[primal stationarity] \label{prop:B_Stationarity}
    Assume that \crefrange{ass:quasi_func1}{ass:cost_func} are satisfied. Then any local minimizer $(\bar y, \bar u) \in H^1_0(\Omega) \times L^p(\Omega)$ of \eqref{eq:P2} satisfies
    \begin{equation}
        \label{eq:B_stationary}
        \partial_y J(\bar y, \bar u)S'(\bar u; h) + \partial_u J(\bar y, \bar u)h \geq 0 \quad \text{for all} \quad h \in L^p(\Omega).
    \end{equation}
\end{proposition}
\begin{proof}
    By virtue of \cref{thm:DD}, the continuous differentiability of the cost functional $J$, and \cite[Lem.~3.9]{Meyer2013}, the reduced cost functional $j: U \to \R$ is directionally differentiable with its directional derivative $\partial_y J(\bar y, \bar u)S'(\bar u; h) + \partial_u J(\bar y, \bar u)h$. The desired result then follows from the local optimality of $\bar u$ and a standard argument.
\end{proof}

In the following subsections, we will derive stronger, dual, optimality conditions that involve Lagrange multipliers for the non-smooth quasilinear equation.

\subsection{C-stationary conditions} \label{sec:reg}

We start with C-stationarity conditions, which can be obtained by regularizing problem \eqref{eq:P2} and passing to the limit.

Let $(\bar y, \bar u) \in H^1_0(\Omega) \times L^p(\Omega)$ be a local minimizer of \eqref{eq:P2} and set $G(u) :=\frac{1}{p} \| u \|^p_{L^p(\Omega)}$. We then consider the regularized problem
\begin{equation} \label{eq:RE_P}
    \tag{P$_\epsilon$} 
    \left\{
        \begin{aligned}
            \min_{u\in L^p(\Omega), y\in H^1_0(\Omega)} &J_\epsilon(y,u):= J(y,u) + G(u-\bar u) \\
            \text{s.t.} \quad
            &\begin{aligned}[t]
                -\div [a_\epsilon(y)\nabla y + b(\nabla y)] &= u &&\text{in } \Omega,\\
                y &= 0 &&\text{on } \partial\Omega,
            \end{aligned}
        \end{aligned}
    \right.
\end{equation}
with $a_\epsilon$ defined in \eqref{eq:a_epsilon}.
The reduced cost functional of \eqref{eq:RE_P} is given by
\begin{equation*}
    j_\epsilon(u):=J_\epsilon(S_{\epsilon}(u),u).
\end{equation*}
In addition, we set
\begin{equation} \label{eq:U_0}
    U_0 := \{ u \in L^p(\Omega) \mid \| u - \bar u \|_{L^p(\Omega)} \leq \bar \rho \},
\end{equation}
Note that $U_0 \subset U$ with $U$ given in \eqref{eq:U_set}.

We first show that any minimizer of \eqref{eq:P2} can be approximated by minimizers of \eqref{eq:RE_P}.
\begin{proposition} \label{prop:conv}
    Assume that \crefrange{ass:quasi_func1}{ass:cost_func} are fulfilled.
    Let $(\bar y, \bar u)\in H^1_0(\Omega)\times L^p(\Omega)$ be a local minimizer of problem \eqref{eq:P2}.
    Then there exists a sequence $\{(y_{\epsilon}, u_\epsilon)\}$ of local minimizers of problems \eqref{eq:RE_P} such that $u_\epsilon \in U_0$, $\{y_\epsilon\}$ is bounded in $W^{1,\bar s}_0(\Omega)$ with $\bar s$ given in \eqref{eq:U_set}, and
    \begin{alignat}{3}
        u_\epsilon &\to \bar u &&\quad \text{in } L^p(\Omega) &&\quad \text{ as } \epsilon \to 0^+, \label{control_conver} \\
        y_{\epsilon} &\to \bar y &&\quad \text{ in } H^1_0(\Omega) \cap C(\overline{\Omega}) &&\quad \text{ as } \epsilon \to 0^+. \label{state_conver}
    \end{alignat}
\end{proposition}
\begin{proof}
    We proceed similarly as in \cite{MeyerSusu2017}.
    Since $(\bar y, \bar u)$ is a local optimal solution to problem \eqref{eq:P2}, there exists $\rho \in (0, \bar \rho)$ such that
    \begin{equation}\label{eq:conv:func}
        j(\bar u) \leq j(u) \qquad\text{for all } u \in L^p(\Omega) \text{ with } \|u - \bar u \|_{L^p(\Omega)} \leq \rho.
    \end{equation}
    We then consider the auxiliary optimal control problem
    \begin{equation} \label{P-auxi}
        \tag{P$_\epsilon^\rho$}
        \min_{\overline B_{L^p(\Omega)}(\bar u, \rho)} j_\epsilon(u),
    \end{equation}
    which by standard arguments admits at least one global minimizer $u_{\epsilon} \in \overline B_{L^p(\Omega)}(\bar u, \rho) \subset U_0$. Now let $\epsilon \to 0^+$. Then there exist a subsequence, denoted by the same symbol, and a function $\tilde{u} \in \overline B_{L^p(\Omega)}(\bar u, \rho)$ such that
    \begin{equation}
        u_{\epsilon} \rightharpoonup \tilde{u} \quad \text{weakly in } L^p(\Omega). \label{limit1}
    \end{equation}
    Combining this with \cref{lem:cont} yields that
    \begin{equation}
        y_{\epsilon} \to S(\tilde{u}) =: \tilde{y} \quad \text{in } H^1_0(\Omega) \cap C(\overline{\Omega}), \label{limit2}
    \end{equation}
    where $y_\epsilon := S_\epsilon(u_\epsilon)$.
    We now show that
    \begin{equation}
        \tilde{u} = \bar u \quad \text{and} \quad u_\epsilon \to \bar u \ \text{in } L^p(\Omega). \label{iden}
    \end{equation}
    In fact, due to $\bar u \in \overline B_{L^p(\Omega)}(\bar u, \rho)$, we have from \cref{lem:cont} that
    \begin{equation}
        j_\epsilon(u_\epsilon) \leq j_\epsilon(\bar u) = J(S_{\epsilon}(\bar u) , \bar u) \to J(S(\bar u), \bar u) = j(\bar u). \label{lim2}
    \end{equation}
    Using the limits \eqref{limit1} and \eqref{limit2} as well as the weak lower semicontinuity of $J$ and of the norm on $L^p(\Omega)$, we arrive at
    \begin{equation}
        \begin{aligned}[t]
            j(\bar u) \geq \limsup_{\epsilon \to 0^+} j_\epsilon(u_\epsilon) &\geq \liminf_{\epsilon \to 0^+} j_\epsilon(u_\epsilon) \\
            & = \liminf_{\epsilon \to 0^+}\left(J(y_{\epsilon} , u_\epsilon) + \frac{1}{p} \| u_\epsilon - \bar u \|_{L^p(\Omega)}^p \right) \\
            & \geq J(\tilde{y} , \tilde{u}) + \frac{1}{p} \| \tilde{u} - \bar u \|_{L^p(\Omega)}^p \\
            & \geq j(\bar u) + \frac{1}{p} \| \tilde{u} - \bar u \|_{L^p(\Omega)}^p.
        \end{aligned}
        \label{lim1}
    \end{equation}
    Here we have just used \eqref{eq:conv:func} to obtain the last inequality in \eqref{lim1}. We thus obtain $\tilde{u} = \bar u$ and $j_\epsilon(u_\epsilon) \to j(\bar u)$.
    We then have
    \begin{equation*}
        \limsup_{\epsilon \to 0^+} \frac{1}{p} \|u_\epsilon - \bar u \|_{L^p(\Omega)}^p = \limsup_{\epsilon \to 0^+} \left( j_\epsilon(u_\epsilon) - J(y_\epsilon , u_\epsilon)\right) \leq 0.
    \end{equation*} Consequently, \eqref{iden} holds. Moreover, the boundedness of $\{y_\epsilon\}$ in $W^{1,\bar s}_0(\Omega)$ follows from the boundedness of $\{u_\epsilon\}$ in $L^p(\Omega)$ and the a priori estimate \eqref{eq:RE_regu}.

    It remains to show that $(y_{\epsilon}, u_\epsilon)$ is a local minimizer of \eqref{eq:RE_P} for sufficiently small $\epsilon >0$. To this end, let $u \in L^p(\Omega)$ be arbitrary with $\| u - u_\epsilon \|_{L^p(\Omega)} < \frac{\rho}{2}$. For $\epsilon$ small enough, we obtain
    \begin{equation*}
        \| u - \bar u \|_{L^p(\Omega)} \leq \| u - u_\epsilon \|_{L^p(\Omega)} + \| \bar u - u_\epsilon \|_{L^p(\Omega)} < \frac{\rho}{2} + \frac{\rho}{2} =\rho,
    \end{equation*}
    which implies that $u$ is a feasible point of problem \eqref{P-auxi}. The global optimality of $u_\epsilon$ for \eqref{P-auxi} thus implies that
    $j_\epsilon(u_\epsilon) \leq j_\epsilon(u). $
    Hence $(y_{\epsilon}, u_\epsilon)$ is a local minimizer of problem \eqref{eq:RE_P}.
\end{proof}

By standard arguments using the continuous differentiability of $J$, the Fr\'{e}chet differentiability of $G$, the G\^{a}teaux differentiability of $S_\epsilon$, and \cref{cor:adjoint_T}, we obtain necessary optimality conditions for the regularized problem \eqref{eq:RE_P}.
\begin{proposition} \label{prop:RE_OS}
    Assume that \crefrange{ass:quasi_func1}{ass:quasi_func2} hold true and that the cost functional $J$ is continuously Fr\'{e}chet differentiable. Then, any local minimizer $(y_{\epsilon}, u_\epsilon)$, $u_\epsilon \in U_0$, of problem \eqref{eq:RE_P} fulfills together with the unique adjoint $w_\epsilon \in H^1_0(\Omega)$ the optimality system
    \begin{subequations}\label{eq:RE_OS}
        \begin{align}
            &\left\{
                \begin{aligned}
                    -\div \left [\left(a_\epsilon(y_\epsilon)\Id + J_b(\nabla y_\epsilon)^T \right)\nabla w_\epsilon \right] + a'_\epsilon(y_\epsilon) \nabla y_\epsilon \cdot \nabla w_\epsilon &= \partial_y J(y_\epsilon,u_\epsilon) && \text{in } \Omega, \\
                    w_\epsilon &=0 && \text{on } \partial\Omega, 
                \end{aligned}
            \right.  \label{eq:RE_adjoint_OS}\\
            &w_\epsilon + \partial_u J(y_\epsilon,u_\epsilon) + G'(u_\epsilon - \bar u) =0. \label{eq:RE_normal_OS}
        \end{align}
    \end{subequations}
\end{proposition}

We now wish to pass to the limit in \eqref{eq:RE_OS}. While this is straightforward for \eqref{eq:RE_normal_OS}, it is difficult to do this in \eqref{eq:RE_adjoint_OS} directly since we have only the boundedness of sequences $\{a'_\epsilon(y_\epsilon)\}$ and $\{\nabla w_\epsilon \}$, respectively, in $L^\infty(\Omega)$ and $L^2(\Omega)$.
Instead, we shall pass to the limit in the \emph{adjoint} equation of \eqref{eq:RE_adjoint_OS} -- which coincides with the linearized equation \eqref{eq:T_oper} -- and apply a duality argument. The presence of Clarke's generalized gradient $\partial_C$ in the following conditions justifies the term \emph{C-stationarity conditions}.
\begin{theorem}[C-stationarity conditions] \label{thm:OS}
    Let \crefrange{ass:quasi_func1}{ass:cost_func} hold and $(\bar y, \bar u)\in H^1_0(\Omega)\times L^p(\Omega)$ be a local minimizer of \eqref{eq:P2}. Then there exist $w \in H^1_0(\Omega)$ and $\chi \in L^\infty(\Omega)$ such that
    \begin{subequations}
        \label{eq:OS}
        \begin{align}
            &\left\{
                \begin{aligned}
                    -\div \left [\left(a(\bar y)\Id + J_b(\nabla \bar y)^T \right)\nabla w \right] + \chi \nabla \bar y \cdot \nabla w &= \partial_y J(\bar y,\bar u) && \text{in } \Omega, \\
                    w &=0 && \text{on } \partial\Omega,
                \end{aligned}
            \right.  \label{eq:adjoint_OS}\\
            &\chi(x) \in \partial_C a\left( \bar y(x) \right) \quad \text{for a.e. } x \in \Omega, \label{eq:Clarke_OS}\\
            &w + \partial_u J(\bar y,\bar u) = 0. \label{eq:normal_OS}
        \end{align}
    \end{subequations}
\end{theorem}
\begin{proof}
    We first address \eqref{eq:normal_OS}. Let $y_\epsilon, u_\epsilon$ and $w_\epsilon$ be given as in \cref{prop:RE_OS,prop:conv}.
    From \eqref{eq:RE_adjoint_OS}, we have $w_\epsilon = ({T}_{y_\epsilon, \epsilon}^{-1})^*\partial_y J(y_\epsilon, u_\epsilon)$. As a result of \cref{prop:conv} and estimate \eqref{eq:T_esti}, we obtain a constant $C > 0$ such that $\| w_\epsilon \|_{H^1_0(\Omega)} \leq C$ for all $\epsilon >0$.
    Then there exists a subsequence, denoted in the same way, satisfying $w_\epsilon \rightharpoonup w$ in $H^1_0(\Omega)$ as $\epsilon \to 0^+$
    for some $w \in H^1_0(\Omega)$. Now, letting $\epsilon \to 0^+$ in \eqref{eq:RE_normal_OS} and using limits \eqref{control_conver} and \eqref{state_conver}, the continuity of $\partial_u J$ gives \eqref{eq:normal_OS}.

    \medskip

    To show \eqref{eq:Clarke_OS}, we first see that there exists a constant $M>0$ such that
    \begin{equation*}
        \|y_\epsilon \|_{C(\overline{\Omega})}, \|\bar y \|_{C(\overline{\Omega})} \leq M \quad \text{for all} \quad \epsilon >0.
    \end{equation*}
    Since $a_\epsilon$ is Lipschitz continuous on $[-M,M]$ with Lipschitz constant $C_{M+1}$, we have
    $|a_\epsilon'(t)| \leq C_{M+1}$ for all $t \in [-M,M]$ and so
    \begin{equation*}
        |a_\epsilon'(y_\epsilon(x))| \leq C_{M+1} \quad \text{for a.e. } x \in \Omega.
    \end{equation*}
    Furthermore, since $a_\epsilon$ is continuously differentiable, $a_\epsilon'(y_\epsilon(x))\in \partial_C a_\epsilon(y_\epsilon(x))$ for almost every $x\in \Omega$.
    We can therefore extract a subsequence, denoted in the same way, such that
    \begin{equation} \label{eq:Clarke_conv}
        a_\epsilon'(y_\epsilon) \rightharpoonup^* \chi \quad \text{in } L^\infty(\Omega) \quad \text{as } \epsilon \to 0^+
    \end{equation}
    for some $\chi\in L^\infty(\Omega)$. 
    Combining this with the limit \eqref{state_conver}, we obtain from \cite[Chap.~I, Thm.~3.14]{Tiba1990} that $\chi(x) \in \partial_C a(\bar y(x))$ for a.e. $x \in \Omega$. This proves \eqref{eq:Clarke_OS}. 

    \medskip

    It remains to show \eqref{eq:adjoint_OS}.
    First, for any $\varphi \in H^{-1}(\Omega)$, we have from the fact that $w_\epsilon = ({T}_{y_\epsilon, \epsilon}^{-1})^*\partial_y J(y_\epsilon, u_\epsilon)$ that
    \begin{equation} \label{eq:adjoint1}
        \left\langle \varphi,w_\epsilon \right\rangle = \left\langle \partial_y J(y_\epsilon, u_\epsilon) , T^{-1}_{y_\epsilon, \epsilon}\varphi \right\rangle.
    \end{equation}
    Thus, $z_\epsilon := T^{-1}_{y_\epsilon, \epsilon}\varphi$ satisfies
    \begin{equation} \label{eq:adjoint2}
        \left\{
            \begin{aligned}
                -\div \left [\left(a_\epsilon(y_\epsilon)\Id + J_b(\nabla y_\epsilon) \right) \nabla z_\epsilon + a'_\epsilon(y_\epsilon) z_\epsilon \nabla y_\epsilon \right] &= \varphi && \text{in } \Omega, \\
                z_\epsilon &=0 && \text{on } \partial\Omega.
            \end{aligned}
        \right.
    \end{equation}
    The a priori estimate \eqref{eq:RE_apriori_DD} (see also \eqref{eq:T_esti}) guarantees the boundedness of $\{z_\epsilon\}$ in $H^{1}_0(\Omega)$. We can thus extract a subsequence, named also by $\{z_\epsilon\}$, such that
    \begin{equation} \label{eq:lim_z_ep}
        z_\epsilon \rightharpoonup z_0 \quad \text{in } H^1_0(\Omega) \quad \text{and} \quad z_\epsilon \to z_0 \quad \text{in } L^{s_1}(\Omega).
    \end{equation}
    Letting $\epsilon \to 0^+$ in \eqref{eq:adjoint2} and using the fact that $a_\epsilon(y_\epsilon) \to a(\bar y)$ in $C(\bar{\Omega})$, $\nabla y_\epsilon \to \nabla \bar y$ in $L^2(\Omega)$, $J_b(\nabla y_\epsilon) \to J_b(\nabla \bar y)$ in $L^m(\Omega)$ for all $m \geq 1$, as well as the limit \eqref{eq:Clarke_conv}, we obtain
    \begin{equation*}
        \left\{
            \begin{aligned}
                -\div \left [\left(a(\bar y)\Id + J_b(\nabla \bar y) \right) \nabla z_0 + \chi z_0 \nabla \bar y \right] &= \varphi && \text{in } \Omega, \\
                z_0 &=0 && \text{on } \partial\Omega.
            \end{aligned}
        \right.
    \end{equation*}
    We now define the operator $\hat T: H^1_0(\Omega) \to H^{-1}(\Omega)$ by
    \begin{equation*}
        \hat T (z) := -\div \left [\left(a(\bar y)\Id + J_b(\nabla \bar y) \right) \nabla z + \chi z \nabla \bar y \right], \quad z \in H^1_0(\Omega).
    \end{equation*}
    Arguing as in the proof of \cref{prop:T_iso} shows that $\hat T$ is an isomorphism.
    In addition, we have $z_0 = \hat T^{-1}(\varphi)$.
    Letting $\epsilon \to 0^+$, the right hand side of \eqref{eq:adjoint1} tends to
    \begin{equation}
        \left \langle \partial_y J(\bar y, \bar u), z_0 \right\rangle = \left\langle \partial_y J(\bar y, \bar u), \hat T^{-1} \varphi \right\rangle = \left\langle \varphi, \left(\hat T^{-1} \right)^*\left( \partial_y J(\bar y, \bar u) \right) \right\rangle.
    \end{equation}
    Here we have used the fact that $\partial_y J(y_\epsilon, u_\epsilon) \to \partial_y J(\bar y, \bar u)$ in $H^{-1}(\Omega)$.
    On the other hand, the left hand side of \eqref{eq:adjoint1} converges to $\left \langle \varphi, w \right\rangle$ as $\epsilon \to 0^+$. We thus have
    \begin{equation*}
        \left \langle \varphi,w \right\rangle = \left\langle \varphi, \left(\hat T^{-1} \right)^*\left( \partial_y J(\bar y, \bar u) \right) \right\rangle.
    \end{equation*}
    Since $\varphi$ is arbitrary in $H^{-1}(\Omega)$, we obtain $w = \left(\hat T^{-1} \right)^*\left( \partial_y J(\bar y, \bar u) \right)$, which yields \eqref{eq:adjoint_OS}.
\end{proof}
Note that the multiplier $\chi\in L^\infty(\Omega)$ is not uniquely determined by \eqref{eq:OS}, which is an obstacle for solving the latter using a Newton-type method. However, under additional assumptions on $a$, we can derive an equivalent optimality system for which this is possible; see \cref{sec:PC1}.

\subsection{Strong stationarity}

The conditions \eqref{eq:OS} can be strengthened by including a pointwise generalized sign condition on the Lagrange multiplier $w$, which are typically referred to as \emph{strong stationarity conditions}.
\begin{theorem}[strong stationarity] \label{thm:strong_stationarity}
    Let \crefrange{ass:quasi_func1}{ass:cost_func} hold and $(\bar y, \bar u)\in H^1_0(\Omega)\times L^p(\Omega)$ be a local minimizer of \eqref{eq:P2}. Then there exist $w \in H^1_0(\Omega)$ and $\chi \in L^\infty(\Omega)$ such that
    \begin{subequations}
        \label{eq:OS_strong}
        \begin{align}
            &\left\{
                \begin{aligned}
                    -\div \left [\left(a(\bar y)\Id + J_b(\nabla \bar y)^T \right)\nabla w \right] + \chi \nabla \bar y \cdot \nabla w &= \partial_y J(\bar y,\bar u) && \text{in } \Omega, \\ 	w &=0 && \text{on } \partial\Omega,
                \end{aligned}
            \right.
            \label{eq:strong_adjoint_OS}\\
            &\chi(x) \in \partial_C a\left( \bar y(x) \right) \quad \text{for a.e. } x \in \Omega, \label{eq:strong_Clarke_OS}\\
            &w + \partial_u J(\bar y,\bar u) = 0, \label{eq:strong_normal_OS} \\
            &\left( a'(\bar y(x); \kappa) - \chi(x)\kappa \right) \nabla \bar y(x) \cdot \nabla w(x) \leq 0 \quad \text{for a.e. }x\in \Omega, \kappa\in \R.
            \label{eq:strong_extra_pw}
        \end{align}
    \end{subequations}
\end{theorem}
\begin{proof}
    Due to \cref{thm:OS}, it only remains to show that \eqref{eq:strong_extra_pw} holds. 
    As a first step, we show that 
    \begin{equation}
        \int_{\Omega} \left( a'(\bar y; z) - \chi z \right) \nabla \bar y \cdot \nabla w dx \leq 0 \quad \forall z \in L^{2}(\Omega)
        \label{eq:strong_extra}
    \end{equation}
    by considering in turn the case of $z \in H^1_0(\Omega)$ with $z := S'(\bar u; h)$ for some $h \in L^p(\Omega)$, followed by the case of  $z \in H^1_0(\Omega)$ arbitrary, and finally the case of $z \in L^{2}(\Omega)$.

    \medskip

    \emph{(i) $z = S'(\bar u; h)$ with $h \in L^p(\Omega)$.} In this case, $z$ satisfies the equation
    \begin{equation*}
        \left\{
            \begin{aligned}
                -\div \left [\left(a(\bar y)\Id + J_b(\nabla \bar y) \right) \nabla z + a'(\bar y; z) \nabla \bar y \right] & = h && \text{in } \Omega, \\
                z &=0 && \text{on } \partial\Omega.
            \end{aligned}
        \right.
    \end{equation*}
    Multiplying the above equation by $w$, integrating over $\Omega$, and using integration by parts, we deduce that
    \begin{equation*}
        \int_{\Omega} \left(a(\bar y)\Id + J_b(\nabla \bar y) \right) \nabla z \cdot \nabla w dx + \int_{\Omega} a'(\bar y; z) \nabla \bar y \cdot \nabla w dx = \int_{\Omega} hw dx,
    \end{equation*}
    which is equivalent to
    \begin{equation*}
        \begin{aligned}
            \left\langle -\div \left [\left(a(\bar y)\Id + J_b(\nabla \bar y)^T \right)\nabla w \right], z \right\rangle + \int_{\Omega} a'(\bar y; z) \nabla \bar y \cdot \nabla w dx = \int_{\Omega} hw dx.
        \end{aligned}
    \end{equation*}
    Combining this with \eqref{eq:strong_adjoint_OS} yields
    \begin{equation*}
        \int_{\Omega} \left( a'(\bar y; z) - \chi z \right) \nabla \bar y \cdot \nabla w dx = - \left\langle \partial_y J(\bar y, \bar u), z \right\rangle + \int_{\Omega} hw dx.
    \end{equation*}
    From this and relation \eqref{eq:strong_normal_OS} as well as the fact that $z = S'(\bar u; h)$, we arrive at
    \begin{equation}
        \label{eq:extra}
        \int_{\Omega} \left( a'(\bar y; z) - \chi z \right) \nabla \bar y \cdot \nabla w dx = - \left\langle \partial_y J(\bar y, \bar u), S'(\bar u; h) \right\rangle - \left\langle \partial_u J(\bar y, \bar u), h \right\rangle \leq 0,
    \end{equation}
    where the last inequality follows from \cref{prop:B_Stationarity}.

    \medskip

    \emph{(ii) $z \in H^1_0(\Omega)$.} Note that $\bar y \in W^{1, \bar s}_0(\Omega) \hookrightarrow C(\overline{\Omega})$, $z \in L^{s_1}(\Omega)$, and so $
    a'(\bar y; z) \in L^{s_1}(\Omega)$ and $a'(\bar y; z) \nabla \bar y \in \left( L^2(\Omega) \right)^N$. Recall that $\bar s$ and $s_1$ are defined as in \eqref{eq:U_set} and \eqref{eq:s1}, respectively. Setting $h:=-\div \left [\left(a(\bar y)\Id + J_b(\nabla \bar y) \right) \nabla z + a'(\bar y; z) \nabla \bar y \right]$, it holds that $h \in H^{-1}(\Omega)$.
    Since $L^p(\Omega)$ is dense in $H^{-1}(\Omega)$ with $p>N/2$, there exists a subsequence $\{h_n\} \subset L^p(\Omega)$ such that $h_n \to h$ in $H^{-1}(\Omega)$. This together with \cref{thm:dire_deri2} implies that $S'(\bar u; h_n) \to z$ in $H^1_0(\Omega)$. Combining this with \eqref{eq:extra}, we obtain that \eqref{eq:strong_extra} holds for $z \in H^1_0(\Omega)$.

    \medskip

    \emph{(iii) $z \in L^{2}(\Omega)$.} In this case, \eqref{eq:strong_extra} is a direct consequence of (ii) and the density of $H^1_0(\Omega)$ in $L^{2}(\Omega)$. 

    \bigskip

    To conclude the proof, we show that \eqref{eq:strong_extra} implies \eqref{eq:strong_extra_pw}. To this end, we assume in contrast that there exist a measurable set $\Omega_0 \subset \Omega$ with $|\Omega_0| >0$ and a number $\kappa \in \R$ such that
    \begin{equation}
        \label{eq:strong_extra1}
        \left( a'(\bar y(x); \kappa) - \chi(x) \kappa \right) \nabla \bar y(x) \cdot \nabla w(x) > 0 \quad \forall x \in \Omega_0.
    \end{equation}
    Setting
    \begin{equation*}
        z(x) :=
        \begin{cases}
            \kappa & \text{if } x \in \Omega_0 \\
            0 & \text{otherwise},
        \end{cases}
    \end{equation*}
    we have that $z \in L^{2}(\Omega)$ and $a'(\bar y(x); z(x)) = 0$ for a.e. $x \in \Omega \setminus \Omega_0$. Estimate \eqref{eq:strong_extra} therefore leads to
    \begin{equation*}
        \int_{\Omega_0} \left( a'(\bar y; \kappa) - \chi \kappa \right) \nabla \bar y \cdot \nabla w \,dx = \int_{\Omega} \left( a'(\bar y; z) - \chi z \right) \nabla \bar y \cdot \nabla w \,dx \leq 0,
    \end{equation*}
    in contradiction to \eqref{eq:strong_extra1}. Hence, \eqref{eq:strong_extra_pw} holds.
\end{proof}
Obtaining more explicit sign conditions on the Lagrange multiplier requires additional assumptions on $a$ (e.g., convexity, which implies regularity of $a$ and hence that $\chi\kappa \leq a'(\bar y(x);\kappa)$ for all $\kappa \in \R$, cf.~\cite[Thm.~4.12]{Constantin2017}) However, as we will show in \cref{prop:equi_SS_OS} below, for a quite general and reasonable class of functions, condition \eqref{eq:strong_extra_pw} holds with equality anyway.

\section{Piecewise differentiable nonlinearities}\label{sec:PC1}

We now consider the special case that the non-smooth nonlinearity $a$ is differentiable apart from countably many points, where it is possible to reformulate the C-stationarity conditions \eqref{eq:OS} in a form that can be solved by a semi-smooth Newton method. 

We first recall the following definition from, e.g. \cite[Chap.~4]{Scholtes} or \cite[Def.~2.19]{Ulbrich2011}.
Let $V$ be an open subset of $\R$. A continuous function $g: V \to \R$ is said to be a \emph{$PC^1$-function} if for each point $t_0 \in V$ there exist a neighborhood $W \subset V$ and a finite set of $C^1$-functions $g_i: W \to \R$, $i=1,2,\dots,m$, such that
\begin{equation*}
    g(t) \in \left\{g_1(t), g_2(t),\dots,g_m(t)\right\} \quad \text{for all} \quad t \in W.
\end{equation*}
For any continuous function $g: V \to \R$, we define the set
\begin{equation*}
    D_{g} := \left\{ t \in V \mid g \ \text{is not differentiable at } t \right\}.
\end{equation*}
We shall say that a $PC^1$ function $g$ is \emph{countably $PC^1$} if the set $D_g$ is countable, i.e., it can be represented as $D_{g} = \{t_i\mid i \in I_g\}$ with a countable set $I$.

\subsection{Optimality conditions}

Under the additional assumption that $a$ is a countably $PC^1$ function, we can show that \eqref{eq:adjoint_OS} holds for \emph{any} $\chi\in L^\infty(\Omega)$ satisfying \eqref{eq:Clarke_OS}.
For this purpose, we introduce for any $y\in H^1_0(\Omega)$ the measurable set
\begin{equation}
    \Omega_y := \{y \in D_a\} = \{y = t_i, i\in I_a\}.
\end{equation}
\begin{theorem}[relaxed optimality system] \label{thm:OS_PC1}
    Assume that \crefrange{ass:quasi_func1}{ass:cost_func} are satisfied and that $a$ is a countably $PC^1$ function. Let $(\bar y, \bar u)\in H^1_0(\Omega)\times L^p(\Omega)$ be a local minimizer of \eqref{eq:P2}. Then there exists a unique $w \in H^1_0(\Omega)$ satisfying 
    \begin{subequations}
        \label{eq:reduced_OS}
        \begin{align}
            &\left\{
                \begin{aligned}
                    -\div \left [\left(a(\bar y)\Id + J_b(\nabla \bar y)^T \right)\nabla w \right] + \chi \nabla \bar y \cdot \nabla w &= \partial_y J(\bar y,\bar u) && \text{in } \Omega, \\
                    w &=0 && \text{on } \partial\Omega, 
                \end{aligned}
            \right. \label{eq:adjoint_OS_PC1}\\
            &w + \partial_u J(\bar y,\bar u) = 0 \label{eq:normal_OS_PC1}
        \end{align}
    \end{subequations}
    for any $\chi \in L^\infty(\Omega)$ with $\chi(x) \in \partial_C a\left( \bar y(x) \right)$ a.e. $x \in \Omega$.
\end{theorem}
\begin{proof}
    In view of \cref{thm:OS}, there exist $w \in H^1_0(\Omega)$ and $\chi_0 \in L^\infty(\Omega)$ with $\chi_0(x) \in \partial_C a\left( \bar y(x) \right)$ for a.e. $x \in \Omega$ satisfying \eqref{eq:normal_OS_PC1} and
    \begin{equation}
        \label{eq:adjoint_OS1}
        \left\{
            \begin{aligned}
                -\div \left [\left(a(\bar y)\Id + J_b(\nabla \bar y)^T \right)\nabla w \right] + \chi_0 \nabla \bar y \cdot \nabla w &= \partial_y J(\bar y,\bar u) && \text{in } \Omega, \\
                w & =0 && \text{on } \partial\Omega.
            \end{aligned}
        \right.
    \end{equation}
    It therefore suffices to prove that \eqref{eq:adjoint_OS1} holds for arbitrary $\chi\in L^\infty(\Omega)$ with $\chi(x) \in \partial_C a\left( \bar y(x) \right)$ for almost every $x\in \Omega$.  

    We now proceed by pointwise almost everywhere inspection.
    \begin{enumerate}[label={Case }\arabic*:, align=left]
        \item $x\in \Omega\setminus\Omega_{\bar y}$. Then by definition, $a$ is differentiable and hence even continuously differentiable in $\bar y(x)$ and hence $\partial_C a(\bar y(x)) = \{a'(\bar y(x))\}$, which implies that $\chi(x) = a'(\bar y(x)) = \chi_0(x)$.
        \item $x\in \Omega_{\bar y}$. Then $x \in \Omega_i:=\{\bar y = t_i\}$ for some $i\in I_a$. Since $\bar y \in H^1_0(\Omega)$, this implies that $\nabla \bar y(x)=0$ for almost every $x\in \Omega_{\bar y}$; see, e.g., \cite[Rem.~2.6]{Chipot2009}. Hence,
            \begin{equation*}
                \chi(x) \nabla \bar y(x)\cdot \nabla w(x) = 0 = \chi_0 \nabla \bar y(x)\cdot \nabla w(x).
            \end{equation*}
    \end{enumerate}
    In both cases, we see that the left-hand side of \eqref{eq:adjoint_OS1} equals that of \eqref{eq:adjoint_OS_PC1}, which yields the claim.
\end{proof}

The benefit of \eqref{eq:reduced_OS} is that we can fix and then eliminate $\chi$ from these optimality conditions such that the reduced system in $(\bar y,\bar u,w)$ has a unique solution, allowing application of a Newton-type method.

Before we turn to this, we note that a similar argument as in the proof of \cref{thm:OS_PC1} shows that under this additional assumption on $a$, strong stationarity is in fact not stronger than C stationarity.
\begin{proposition}[strong stationarity reduces to C-stationarity]
    \label{prop:equi_SS_OS}
    Assume that $a$ is countably $PC^1$. Then, the system \eqref{eq:OS_strong} is equivalent to the system \eqref{eq:OS}.
\end{proposition}
\begin{proof}
    Clearly, \eqref{eq:OS_strong} implies \eqref{eq:OS}. It is therefore sufficient to show that if \eqref{eq:OS} holds, the sign condition \eqref{eq:strong_extra_pw} is always satisfied. We again proceed by pointwise almost everywhere inspection.
    \begin{enumerate}[label={Case }\arabic*:, align=left]
        \item $x\in \Omega\setminus\Omega_{\bar y}$. Then by definition, $a$ is differentiable and hence even continuously differentiable in $\bar y(x)$ and hence $\partial_C a(\bar y(x)) = \{a'(\bar y(x))\}$, which implies that $\chi(x)z(x) = a'(\bar y(x))z(x) = a'(\bar y(x); z(x))$.
        \item $x\in \Omega_{\bar y}$. Then $\nabla \bar y(x) = 0$ as above.
    \end{enumerate}
    In both cases, we obtain that \eqref{eq:strong_extra_pw} holds with equality.
\end{proof}

\subsection{Semi-smooth Newton method}\label{sec:semi-smooth}

For the sake of presentation,
we consider problem \eqref{eq:P2} for $a(y) := 1 + |y|$, $b \equiv 0$, and 
\begin{equation*}
    J(y,u) := \frac{1}{2} \|y - y_d \|^2_{L^2(\Omega)} + \frac{\alpha}{2}\|u \|^2_{L^2(\Omega)},
\end{equation*}
where $\Omega$ is a bounded convex domain in $\R^N$, $N \in \{2,3\}$, $y_d$ is a given desired function in $L^\infty(\Omega)$, and $\alpha>0$. 

Let $(\bar y,\bar u)\in H^1_0(\Omega)\times L^2(\Omega)$ be a local minimizer for this instance of problem \eqref{eq:P2}, and let $\bar \chi\in L^\infty(\Omega)$ with $\bar \chi(x) \in \partial_C a(\bar y(x))=\sign(\bar y(x))$ for a.e. $x\in \Omega$ be arbitrary but fixed, e.g.,
\begin{equation*}
    \bar \chi(x) = 
    \begin{cases}
        1 & \text{if } \bar y(x) \geq 0,\\
        -1 & \text{if } \bar y(x) <0.
    \end{cases}
\end{equation*}
We then consider the system
\begin{equation}
    \label{eq:OS_exam}
    \left\{
        \begin{aligned}
            - \div \left[ \left( 1 + |\bar y| \right) \nabla \bar y \right] &= \bar u && \text{in } \Omega, \qquad \bar y = 0 \quad \text{on } \partial\Omega,\\
            - \div \left[ \left( 1 + |\bar y| \right) \nabla w \right] + \bar \chi \nabla \bar y \cdot \nabla w &= \bar y - y_d && \text{in } \Omega, \qquad w = 0 \quad \text{on } \partial\Omega,\\
            w + \alpha \bar u &= 0, &&
        \end{aligned}
    \right.
\end{equation}
which by \cref{thm:OS_PC1} admits a solution $(\bar y,\bar u,w)\in H^1_0(\Omega)\times L^2(\Omega)\times H^1_0(\Omega)$. By virtue of \cref{lem:regularity}, it holds that $\bar y \in H^2(\Omega) \cap H^1_0(\Omega)$, which implies that $1+|\bar y| \in W^{1,4}(\Omega)$ and $\bar\chi\nabla \bar y \in L^4(\Omega)^N$. From this and \cite[Thm.~2.6]{CasasDhamo2011}, we deduce that $w \in H^2(\Omega) \cap H^1_0(\Omega)$ as well. The second equation in \eqref{eq:OS_exam} is then equivalent to
\begin{equation}
    \label{eq:OS_equ2}
    \left\{
        \begin{aligned}
            -\Delta  w - \frac{\bar y-y_d}{1 + |\bar y|} &= 0 && \text{in } \Omega, \\
            w &= 0 && \text{on } \partial\Omega.
        \end{aligned}
    \right.
\end{equation}
Setting $\psi := \bar y + \frac{\bar y|\bar y|}{2}$, a simple computation shows that
\begin{equation}
    \label{eq:represent_y}
    \bar y = \left(-1 + \sqrt{1 + 2\psi} \right) \1_{\left\{ \psi \geq 0 \right\} } + \left(1 - \sqrt{1 - 2\psi} \right) \1_{\left\{ \psi < 0 \right\} }.
\end{equation}
By eliminating the control $\bar u$ using the third equation in \eqref{eq:OS_exam} and using \eqref{eq:OS_equ2} together with \eqref{eq:represent_y}, we see that \eqref{eq:OS_exam} is equivalent to
\begin{equation}
    \label{eq:OS_NE}
    \left\{
        \begin{aligned}
            - \Delta \psi + \frac{1}{\alpha} w &= 0 && \text{in } \Omega, \qquad \psi = 0 \quad \text{on } \partial\Omega,\\
            - \Delta w - \left[ f_1(\psi) - y_d f_2(\psi) \right] &= 0 && \text{in } \Omega, \qquad w = 0 \quad \text{on } \partial\Omega,
        \end{aligned}
    \right.
\end{equation}
for
\begin{equation*}
    \begin{aligned}
        f_1(\psi) &:= \frac{-1 + \sqrt{1 + 2\psi}}{\sqrt{1 + 2\psi}}\1_{\left\{ \psi \geq 0 \right\} } + \frac{1 - \sqrt{1 - 2\psi}}{\sqrt{1 - 2\psi}}\1_{\left\{ \psi < 0 \right\} } = \frac{-1 + \sqrt{1 + 2|\psi|}}{\sqrt{1 + 2|\psi|}}\sign(\psi), \\
        f_2(\psi) &:= \frac{1}{\sqrt{1 + 2\psi}}\1_{\left\{ \psi \geq 0 \right\} } + \frac{1}{\sqrt{1 - 2\psi}}\1_{\left\{ \psi < 0 \right\} } = \frac{1}{\sqrt{1 + 2|\psi|}}.
    \end{aligned}
\end{equation*}
(Note that $f_1(0)=0$ and hence is single-valued in spite of the occurrence of the set-valued $\sign$.)
Since $f_1$ and $f_2$ are globally Lipschitz continuous and $PC^1$-functions and $y_d \in L^\infty(\Omega)$, the corresponding superposition operators are semi-smooth as functions from $H^1_0(\Omega)$ to $L^2(\Omega)$; see, e.g., \cite[Thm.~3.49]{Ulbrich2011}. From this and the continuous embedding $L^2(\Omega) \hookrightarrow H^{-1}(\Omega)$, we conclude that the system \eqref{eq:OS_NE} is semi-smooth as an equation from $H^1_0(\Omega) \times H^1_0(\Omega)$ to $H^{-1}(\Omega) \times H^{-1}(\Omega)$

Introducing for $k\in \N$ the pointwise multiplication operator $D_k:H^1_0(\Omega)\to H^{-1}(\Omega)$ with 
\begin{equation}
    \label{eq:hk_formula}
    d^k  := \frac{1}{\left(\sqrt{1 + 2\psi^k}\right)^3}\1_{\left\{ \psi^k \geq 0 \right\}}(1+y_d) + \frac{1}{\left(\sqrt{1 - 2\psi^k}\right)^3}\1_{\left\{ \psi^k < 0 \right\}}(1-y_d),
\end{equation}
a semi-smooth Newton step thus consists in solving
\begin{equation} \label{eq:N_step}
    \begin{pmatrix}
        \frac{1}{\alpha} \Id & - \Delta \\
        -\Delta & - D_k
    \end{pmatrix}
    \begin{pmatrix}
        \delta w \\
        \delta \psi
    \end{pmatrix}  = -
    \begin{pmatrix}
        - \Delta \psi^k + \frac{1}{\alpha} w^k \\
        -\Delta w^k - \left[ f_1(\psi^k) - y_d f_2(\psi^k) \right]
    \end{pmatrix}
\end{equation}
and setting $(w^{k+1}, \psi^{k+1}) := (w^k, \psi^k) + (\delta w, \delta \psi)$. 
To show the local superlinear convergence of this iteration, it is sufficient to prove the uniformly bounded invertibility of \eqref{eq:N_step}. To this end, we employ a technique as in \cite{ClasonKunisch2016}. 
\begin{lemma}
    \label{lem:Newton}
    Assume that $d \in L^\infty(\Omega)$ such that $\| d\|_{L^\infty(\Omega)} \leq M$ for some $M >0$. If either $d \geq 0$ a.e. or $\alpha > 0$ is large enough, then the operator $B: H^1_0(\Omega)^2 =: E \to F := H^{-1}(\Omega)^2$ given by
    \begin{equation*}
        B :=
        \begin{pmatrix}
            \frac{1}{\alpha} \Id & - \Delta \\
            -\Delta & - D
        \end{pmatrix},
    \end{equation*}
    where $D$ is the pointwise multiplication operator with $d$,
    is uniformly invertible and satisfies
    \begin{equation*}
        \left \| B^{-1} \right \|_{\mathcal{L}(F, E)} \leq C
    \end{equation*}
    for some constant $C$ depending only on $\alpha$ and $M$.
\end{lemma}
\begin{proof}
    Let $r_1$ and $r_2$ be arbitrary in $H^{-1}(\Omega)$ and consider the equation
    $B \begin{psmallmatrix} \delta w \\ \delta \psi \end{psmallmatrix} = \begin{psmallmatrix} r_1 \\ r_2 \end{psmallmatrix}$, i.e.,
    \begin{equation}
        \label{eq:auxi2}
        \left \{
            \begin{aligned}
                - \Delta \delta \psi + \frac{1}{\alpha} \delta w & = r_1 && \text{in } \Omega, \qquad \delta\psi = 0 \quad \text{on } \partial\Omega \\
                - \Delta \delta w - d \delta \psi &= r_2 && \text{in } \Omega, \qquad \delta w = 0 \quad \text{on } \partial\Omega.
            \end{aligned}
        \right.
    \end{equation}
    Obviously, $-\Delta$ is isomorphic as an operator from $H^1_0(\Omega)$ to $H^{-1}(\Omega)$, and $(-\Delta)^{-1}: L^2(\Omega) \to L^2(\Omega)$ is self-adjoint. We now consider the continuous bilinear form $e: L^2(\Omega) \times L^2(\Omega) \to \R$ defined via
    \begin{equation*}
        e(w,v) = (w,v)_{L^2(\Omega)} + \frac{1}{\alpha} \left( d (-\Delta)^{-1} w, (-\Delta)^{-1} v \right)_{L^2(\Omega)}, \quad \text{for all }w, v \in L^2(\Omega),
    \end{equation*}
    where $(\cdot, \cdot)_{L^2(\Omega)}$ stands for the inner product in $L^2(\Omega)$. 
    We now show that $e$ is coercive, i.e., there exists a constant $\lambda >0$ such that
    \begin{equation}
        \label{eq:coercive}
        e(w,w) \geq \lambda \|w \|_{L^2(\Omega)}^2 \quad \text{for all } w \in L^2(\Omega).
    \end{equation}
    If $d \geq 0$ almost everywhere, then \eqref{eq:coercive} holds with $\lambda = 1$. It therefore remains to prove \eqref{eq:coercive} for the case where $\alpha$ is large enough.
    To this end, we observe for any $w \in L^2(\Omega)$ that
    \begin{equation*}
        \begin{aligned}
            e(w,w) &\geq \|w \|_{L^2(\Omega)}^2 - \frac{1}{\alpha} \| d\|_{L^\infty(\Omega)} \| (-\Delta)^{-1}w \|_{L^2(\Omega)}^2\\
            & \geq \|w \|_{L^2(\Omega)}^2 \left(1 - \frac{C_0^2M}{\alpha} \right)
        \end{aligned}
    \end{equation*}
    with $C_0 := \| (-\Delta)^{-1} \|_{\mathcal L(L^2(\Omega))}$. This yields \eqref{eq:coercive} provided that $\alpha > C_0^2M$.

    The Lax--Milgram theorem now implies that there exist a unique $\tilde{w} \in L^2(\Omega)$ and a constant $C = C(\alpha, M)>0$ such that
    \begin{equation}
        \label{eq:auxi}
        e(\tilde{w}, v) = \left( (-\Delta)^{-1}(r_2 + d (-\Delta)^{-1}r_1), v \right)_{L^2(\Omega)} \quad \text{for all} \quad v \in L^2(\Omega)
    \end{equation}
    and
    \begin{equation}\label{eq:auxi_apriori}
        \| \tilde{w} \|_{L^2(\Omega)} \leq C \left( \| r_1 \|_{H^{-1}(\Omega)} + \| r_2 \|_{H^{-1}(\Omega)} \right).
    \end{equation}
    Let $\delta\psi\in H^1_0(\Omega)$ be the solution to
    \begin{equation}
        \label{eq:auxi3}
        - \Delta \delta \psi + \frac{1}{\alpha} \tilde{w} = r_1 \quad \text{in } \Omega, \quad \delta \psi = 0 \quad \text{on } \partial\Omega,
    \end{equation}
    and let $\delta w\in H^1_0(\Omega)$ be the corresponding solution to the second equation in \eqref{eq:auxi2}. Then it follows from \eqref{eq:auxi_apriori} that
    \begin{equation*}
        \| \delta w \|_{H^1_0(\Omega)} + \| \delta\psi \|_{H^1_0(\Omega)} \leq C \left( \| r_1 \|_{H^{-1}(\Omega)} + \| r_2 \|_{H^{-1}(\Omega)} \right).
    \end{equation*}

    Finally, it follows from \eqref{eq:auxi3}, the second equation in \eqref{eq:auxi2}, and \eqref{eq:auxi}
    \begin{equation*}
        \begin{aligned}
            \delta w &= (-\Delta)^{-1} \left[ r_2 + d \delta \psi \right]\\
            & = (-\Delta)^{-1} \left[ r_2 + d (-\Delta)^{-1} \left( r_1 - \frac{1}{\alpha} \tilde{w} \right) \right] \\
            & = (-\Delta)^{-1} \left[ r_2 + d (-\Delta)^{-1} r_1 \right] - \frac{1}{\alpha}(-\Delta)^{-1} \left[ d (-\Delta)^{-1} \tilde{w} \right] \\
            & = \tilde{w},
        \end{aligned}
    \end{equation*}
    which concludes the proof.
\end{proof}

We now arrive at the local convergence of the semi-smooth Newton iteration.
\begin{theorem} \label{thm:Newton}
    Let $y_d \in L^\infty(\Omega)$ and $\alpha >0$ be such that either $\| y_d \|_{L^\infty(\Omega)} \leq 1$ or $\alpha$ is large enough. Assume that $(w, \psi)$ is a solution to \eqref{eq:OS_NE}. Then there exists a constant $\rho> 0$ such that for $w^0 \in B_{H^1_0(\Omega)}( w, \rho)$ and $\psi^0 \in B_{H^1_0(\Omega)}(\psi, \rho)$, the semi-smooth Newton iteration \eqref{eq:N_step} converges superlinearly in $H^1_0(\Omega) \times H^1_0(\Omega)$ to $(w, \psi)$.
\end{theorem}
\begin{proof}
    From \eqref{eq:hk_formula}, we obtain that $\| d^k\|_{L^\infty(\Omega)} \leq 1 + \| y_d \|_{L^\infty(\Omega)}$ and $d^k \geq 0$ a.e. if $\| y_d \|_{L^\infty(\Omega)} \leq 1$. The claim then follows from \cref{lem:Newton} together with \cite[Thm.~8.16]{ItoKunisch2008}.
\end{proof}

\section{Numerical experiments} \label{sec:numerical_experiment}

We now illustrate the solvability of the relaxed optimality system \eqref{eq:reduced_OS} using the semi-smooth Newton method presented in \cref{sec:semi-smooth} using a numerical example.

Specifically, we consider $\Omega=(0,1)^2\subset \R^2$ and create a uniform triangular Friedrichs--Keller triangulation with $n_h\times n_h$ vertices which is the basis for a finite element discretization of the two elliptic equations in \eqref{eq:OS_NE}. We then compute a solution $(w_h, \psi_h)$ from \eqref{eq:OS_NE} by the presented semi-smooth Newton iteration, starting with $(w^0,\psi^0)=(0,0)$ and terminating whenever the number of iterations reaches $25$ or the active sets $\{\psi^k\geq 0\}$ corresponding to two consecutive steps coincide. If the iteration is successful, we recover $y_h$ via \eqref{eq:represent_y} and $u_h$ via the third equation of \eqref{eq:OS_exam}. The Python implementation using DOLFIN \cite{LoggWells2010a,LoggWellsEtAl2012a} that was used to generate the following results can be downloaded from \url{https://github.com/clason/nonsmoothquasilinear}.

We choose the target function $y_d$ based on a constructed example, setting
\begin{align*}
    \bar y &= x_1^4 \left[ (x_1 -\beta)^4 + 2(x_1 - \beta)^5 \right] \sin(\pi x_2)\1_{[0,\beta]}(x_1),\\
    \bar u &= -(1+|\bar y|)\Delta \bar y - \sign(\bar y)\left|\nabla \bar y\right|^2, \\
    \bar w &= -\alpha \bar u, \\
    \bar \psi &= \bar y \left( 1 + \frac{1}{2} |\bar y | \right),\\
    y_d &= \bar y + (1+|\bar y|)\Delta \bar w, 
\end{align*}
for a parameter $\beta \in [0.5,1]$; see \cref{fig:state_adjoint} for $\alpha = 10^{-7}$ and $\beta = 0.85$.
\begin{figure}[t]
    \centering
    \begin{subfigure}[t]{0.49\textwidth}
        \centering
        \includegraphics[width=\linewidth]{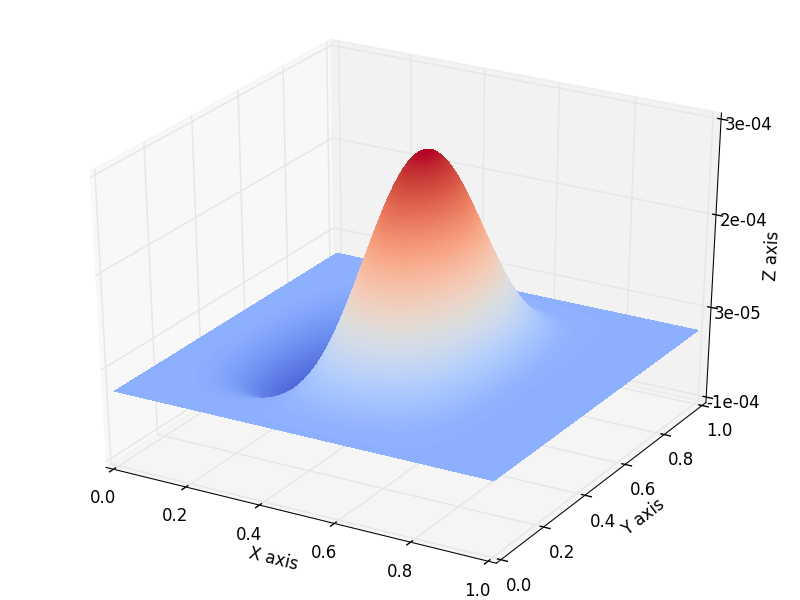}
        \caption{exact state $\bar y$}
        \label{fig:ex_state}
    \end{subfigure}
    \begin{subfigure}[t]{0.49\textwidth}
        \centering
        \includegraphics[width=\linewidth]{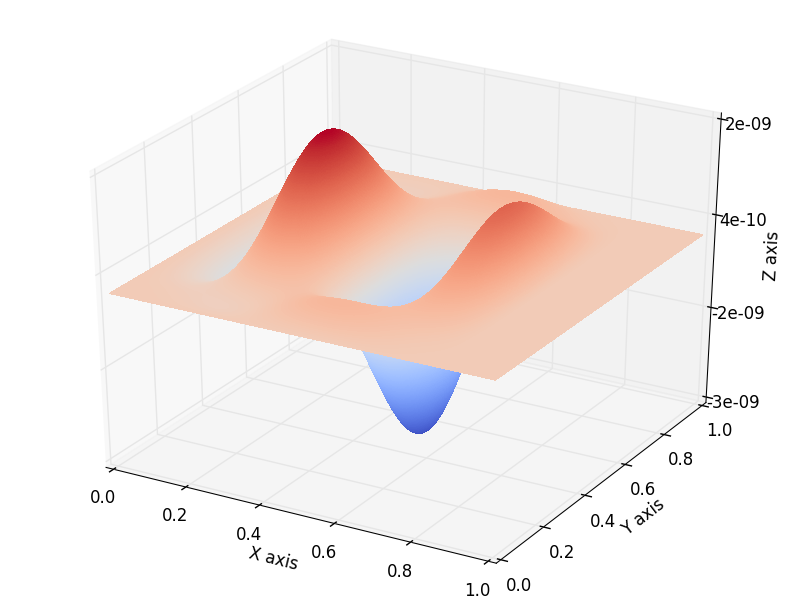}
        \caption{exact adjoint $\bar w=-\alpha\bar u$}
        \label{fig:ex_adjoint}
    \end{subfigure}
    \caption{constructed exact solution for $\alpha = 10^{-7}$, $\beta = 0.85$}
    \label{fig:state_adjoint}
\end{figure}
We point out that for $\beta \in (0.5,1)$, the sets on which the value of $\bar y$ is positive, negative, and zero all have positive measure. Moreover, 
\begin{equation*}
    \left | \left\{ \bar y = 0 \right\} \right|= \left | \left\{ \bar \psi = 0 \right\} \right| = 1 - \beta,
\end{equation*}
and $\|y_d\|_{L^\infty(\Omega)} \leq 1$ for any $\alpha$ small enough and for any $\beta \in [0.5,1]$. From this and \cref{thm:Newton}, we can deduce that the semi-smooth Newton method for solving \eqref{eq:OS_NE} will converge locally superlinearly.

The results for different values of $n_h$, $\alpha$, and $\beta$ are given in \cref{tab:convergence}, where we list the relative $H^1_0$ errors for the computed state $y_h$ and adjoint $w_h$ as well as the number of semi-smooth Newton iterations. For the sake of completeness, we also give the $L^\infty$ norm of $y_d$ for the chosen parameters, verifying that $\|y_d\|_{L^\infty(\Omega)}\leq 1$ in all cases.
We first address the dependence on $n_h$. As can be seen from \cref{tab:convergence:mesh}, the relative errors decrease linearly with increasing $n_h$, which matches the expected $\mathcal{O}(h)$ convergence of the piecewise linear finite element approximation. Also, the number of semi-smooth Newton iterations ($2$--$4$) stays constant, demonstrating the mesh independence that is usually the consequence of an infinite-dimensional convergence result like \cref{thm:Newton}.
The dependence on $\alpha$ is shown in \cref{tab:convergence:alpha}. We point out that the Newton method is relatively robust with respect to this parameter with the required number of iterations only starting to increase from $3$ to $25$ for $\alpha<10^{-6}$.
Finally, we comment on the dependence of $\beta$ shown in \cref{tab:convergence:beta}. Since meas$\{\bar y =0 \} \to 0$ for $\beta \to 1$, it is not surprising that the relative errors for $\bar y$ decrease quickly as $\beta$ increases. Here we observe only a slight increase in the number of semi-smooth Newton iterations from $3$ to $6$.
\begin{table}[t]
    \centering
    \begin{subtable}[t]{\textwidth}
        \centering
        \begin{tabular}{%
                S[table-format=4]
                S[table-format=1e-1,scientific-notation=true,round-mode=places,round-precision=2]
                S[table-format=1.1]
                S[table-format=1.2e-1,scientific-notation=true,round-mode=places,round-precision=2]
                S[table-format=1.2e-1,scientific-notation=true,round-mode=places,round-precision=2]
                S[table-format=2]
                S[table-format=1.2e-1,scientific-notation=true,round-mode=places,round-precision=2]
            }
            \toprule
            {$n_h$} & {$\alpha$} & {$\beta$} & {$\frac{\| y_h - \bar y\|_{H^1_0(\Omega)}}{\|\bar y\|_{H^1_0(\Omega)}}$} & {$\frac{\| w_h - \bar w\|_{H^1_0(\Omega)}}{\|\bar w\|_{H^1_0(\Omega)}}$} & {\#\,SSN} & {$\|y_d \|_{L^\infty(\Omega)} $} \\
            \midrule
            100	& 1e-6 & 0.8 & 0.0032749713163498563		& 0.029153100189308785 	& 2 & 0.000206922267655 \\
            200	& 1e-6 & 0.8 & 0.0016603040704769316		& 0.015400971381311359 	& 4 & 0.000207041692934 	\\
            400	& 1e-6 & 0.8 & 0.0008357366096390089		& 0.007924652139925669 	& 3 & 0.000207041692934	\\
            800	& 1e-6 & 0.8 & 0.0004192504175491312		& 0.004026679039232571 	& 3 & 0.000207054280755 	\\
            1000& 1e-6 & 0.8 & 0.0003356296997407327		& 0.0032365982649012522 	& 3 & 0.000207054028385	\\
            \bottomrule
        \end{tabular}
        \caption{dependence on $n_h$}\label{tab:convergence:mesh}
    \end{subtable}

    \medskip

    \begin{subtable}[t]{\textwidth}
        \centering
        \begin{tabular}{%
                S[table-format=3]
                S[table-format=1e-1,scientific-notation=true,round-mode=places,round-precision=2]
                S[table-format=1.1]
                S[table-format=1.2e-1,scientific-notation=true,round-mode=places,round-precision=2]
                S[table-format=1.2e-1,scientific-notation=true,round-mode=places,round-precision=2]
                S[table-format=2]
                S[table-format=1.2e-1,scientific-notation=true,round-mode=places,round-precision=2]
            }
            \toprule
            {$n_h$} & {$\alpha$} & {$\beta$} & {$\frac{\| y_h - \bar y\|_{H^1_0(\Omega)}}{\|\bar y\|_{H^1_0(\Omega)}}$} & {$\frac{\| w_h - \bar w\|_{H^1_0(\Omega)}}{\|\bar w\|_{H^1_0(\Omega)}}$} & {\#\,SSN} & {$\|y_d \|_{L^\infty(\Omega)} $} \\
            \midrule
            800	& 1e-2 & 0.8	& 0.0635839573071		& 0.0135972357471 	& 4	 & 0.098304 			\\
            800	& 1e-4 & 0.8 	& 0.00876247105136		& 0.00732413971746 	& 3 & 0.00098304 			\\
            800	& 1e-6 & 0.8 	& 0.000419250417549 	& 0.00402667903923 	& 3 & 0.000207054280755 	\\
            800	& 1e-8 & 0.8 	& 2.3210764209021246e-05& 0.002191885233352805 & 25 & 0.0002026756942 	\\
            \bottomrule
        \end{tabular}
        \caption{dependence on $\alpha$}\label{tab:convergence:alpha}
    \end{subtable}

    \medskip

    \begin{subtable}[t]{\textwidth}
        \centering
        \begin{tabular}{%
                S[table-format=3]
                S[table-format=1e-1,scientific-notation=true,round-mode=places,round-precision=2]
                S[table-format=1.1]
                S[table-format=1.2e-1,scientific-notation=true,round-mode=places,round-precision=2]
                S[table-format=1.2e-1,scientific-notation=true,round-mode=places,round-precision=2]
                S[table-format=2]
                S[table-format=1.2e-1,scientific-notation=true,round-mode=places,round-precision=2]
            }
            \toprule
            {$n_h$} & {$\alpha$} & {$\beta$} & {$\frac{\| y_h - \bar y\|_{H^1_0(\Omega)}}{\|\bar y\|_{H^1_0(\Omega)}}$} & {$\frac{\| w_h - \bar w\|_{H^1_0(\Omega)}}{\|\bar w\|_{H^1_0(\Omega)}}$} & {\#\,SSN} & {$\|y_d \|_{L^\infty(\Omega)} $} \\
            \midrule
            800	& 1e-5 & 0.5 & 0.00702996981895		& 0.0119819234703 	& 3 & 1.5e-05 			\\
            800	& 1e-5 & 0.7 & 0.00268166116258 		& 0.00710685957014 	& 3 & 0.000107009251408 	\\
            800	& 1e-5 & 0.9 & 0.0014146893988			& 0.00426761449113 	& 4 & 0.000507493724079 	\\
            800	& 1e-5 & 1.0 & 8.64719106979e-05 		& 0.00338749750832 	& 6 & 0.000949785024956 	\\
            \bottomrule
        \end{tabular}
        \caption{dependence on $\beta$}\label{tab:convergence:beta}
    \end{subtable}
    \caption{numerical results: number of Newton iterations and relative errors for state $\bar y$ and adjoint $\bar w$ in dependence of $n_h$, $\alpha$, and $\beta$}
    \label{tab:convergence}
\end{table}

\section{Conclusions}

We have considered optimal control problems for a quasilinear elliptic differential equation with a nonlinear coefficient in the leading term that is Lipschitz continuous and directionally but not Gâteaux differentiable. By passing to the limit in a regularized equation, C- and strong stationarity conditions can be derived. If the nonlinear coefficient is a piecewise differentiable apart from a countable set of points, both stationarity conditions coincide and are equivalent to a relaxed optimality system that is amenable to numerical solution by a semi-smooth Newton method. This is illustrated by a numerical example.

This work can be extended in several directions. First, second-order sufficient optimality conditions can be considered based on the approach in \cite{Livia2018} for an optimal control problem of non-smooth, semilinear parabolic equations. Furthermore, a practically relevant issue would be to derive error estimates for the finite element approximation of \eqref{eq:P} as in the case of smooth settings \cite{CasasDhamo2012,CasasTroltzsch2011}. Finally, the results derived in this work can be used to study the Tikhonov regularization of parameter identification problems for non-smooth quasilinear elliptic equations.

\appendix

\section{Auxiliary lemmas}\label{appendix}

The first lemma about monotonicity of an auxiliary problem is needed in \cref{thm:regu} to show higher regularity of the state equation \eqref{eq:state}.
\begin{lemma} \label{lem:mono}
    Assume that \cref{ass:quasi_func2} is fulfilled and $\lambda > 0$. Then the operator 
    \begin{equation}
        B: H^1_0(\Omega) \to H^{-1}(\Omega), \qquad B(z) := -\div\left[\lambda \nabla z + b(\nabla z)\right], 
    \end{equation}
    is maximally monotone.
\end{lemma}
\begin{proof}
    Obviously, $B$ is monotone since $b$ is monotone.
    We now show that $B$ is hemicontinuous, i.e., the mapping $[0,1] \ni t \mapsto \langle B(z_1 + tz_2), z_3 \rangle \in \R$ is continuous for all $z_1,z_2,z_3 \in H^1_0(\Omega)$. To this end, observe that for any $z_1,z_2,z_3 \in H^1_0(\Omega)$ and any $t \in [0,1]$,
    \begin{equation}
        \notag
        \langle B(z_1 + tz_2), z_3 \rangle = \lambda \int_{\Omega} \left( \nabla z_1 + t \nabla z_2 \right) \cdot \nabla z_3 dx + \int_{\Omega} b\left( \nabla z_1 + t \nabla z_2 \right) \cdot \nabla z_3 dx.
    \end{equation}
    Together with the continuity of $b$, this implies the hemicontinuity of $B$. The maximal monotonicity of $B$ then follows from \cite[Prop.~32.7]{Zeidler2B}.
\end{proof}

The next lemma shows $H^2$-regularity of the solutions to the state equation with a $PC^1$ nonlinearity and is needed in \cref{sec:semi-smooth} to show Newton differentiability of the relaxed optimality system \eqref{eq:OS_exam}.
\begin{lemma} \label{lem:regularity}
    Let $\Omega$ be a convex domain in $\R^{N}$ with $N \in \{2,3\}$.
    Assume that \cref{ass:quasi_func1} is valid. Assume furthermore that $a$ is a $PC^1$-function. Then, for each $u \in L^2(\Omega)$, the equation
    \begin{equation} \label{eq:state_num}
        \left \{
            \begin{aligned}
                -\div [a(y)\nabla y ] &= u && \text{in } \Omega, \\
                y &=0 && \text{on } \partial\Omega
            \end{aligned}
        \right.
    \end{equation} has a unique solution $y\in H^2(\Omega) \cap H^1_0(\Omega)$.
\end{lemma}
\begin{proof}
    In view of \cref{thm:existence}, it suffices to prove the $H^2$-regularity of the unique solution $y \in  H^1_0(\Omega)\cap C(\overline\Omega)$ of \eqref{eq:state_num}.
    Setting
    \begin{equation*}
        \theta:= K(y) := \int_0^y a(t)dt,
    \end{equation*}
    equation \eqref{eq:state_num} reduces to 
    \begin{equation*}
        \left \{
            \begin{aligned}
                -\Delta \theta &= u && \text{in } \Omega, \\
                \theta &=0 && \text{on } \partial\Omega.
            \end{aligned}
        \right.
    \end{equation*}
    The regularity of solutions to Poisson's equation guarantees that $\theta \in H^2(\Omega)$; see e.g., \cite[Thm.~3.2.1.2]{Grisvard1985}. Since $y = K^{-1}(\theta)$,
    we have that
    \begin{equation*}
        \frac{\partial y}{\partial x_i} = \frac{1}{K'(K^{-1}(\theta))}\frac{\partial \theta}{\partial x_i} = \frac{1}{a(y)}\frac{\partial \theta}{\partial x_i},
    \end{equation*}
    which implies that
    $\| y\|_{H^1_0(\Omega)} \leq \frac{1}{a_0}\| \theta\|_{H^1_0(\Omega)}$. We furthermore have that
    \begin{equation*}
        \frac{\partial^2 y}{\partial x_j\partial x_i} = \frac{1}{a(y)}\frac{\partial^2 \theta}{\partial x_j\partial x_i} - \frac{1}{a^2(y)}\frac{\partial \theta}{\partial x_i}\frac{\partial a(y)}{\partial x_j}.
    \end{equation*}
    Note that, since $y \in C(\overline{\Omega})$, there exists a constant $M>0$ such that $|y(x)| \leq M$ for all $x \in \overline{\Omega}$. Defining a $PC^1$-function 
    \begin{equation*}
        a_M: \R \to \R, \qquad   
        a_M(t) = 
        \begin{cases}
            a(2M) & \text{if } t>2M,\\
            a(t) & \text{if }|t|\leq 2M,\\
            a(-2M) & \text{if } t<-2M.
        \end{cases}
    \end{equation*}
    \cref{ass:quasi_func1} implies that $a_M$ is Lipschitz continuous with Lipschitz constant $C_{2M}$. We then have $\| \nabla a_M \|_{L^\infty(\R)} \leq C_{2M}$, where $\nabla a_M $ is the weak derivative of $a_M$. From this and the chain rule (see, e.g. \cite[Thm.~7.8]{Gilbarg_Trudinger}), we arrive at $a(y) \in H^1(\Omega)$ and
    \begin{equation*}
        \frac{\partial a(y)}{\partial x_j}(x) =
        \begin{cases}
            a'(y(x))\frac{\partial y}{\partial x_j}(x) & \text{if } y(x) \notin D_a, \\
            0 & \text{otherwise}.
        \end{cases}
    \end{equation*}
    Consequently, we have
    \begin{equation*}
        \frac{\partial^2 y}{\partial x_j\partial x_i} = \frac{1}{a(y)}\frac{\partial^2 \theta}{\partial x_j\partial x_i} - \frac{\nabla a_M(y)}{a^3(y)}\frac{\partial \theta}{\partial x_i}\frac{\partial \theta}{\partial x_j} \1_{ \{ y \notin D_a \}}
    \end{equation*}
    and hence
    \begin{equation*}
        \left\| \frac{\partial^2 y}{\partial x_j\partial x_i} \right\|_{L^2(\Omega)} \leq \frac{1}{a_0} \left\| \frac{\partial^2 \theta}{\partial x_j\partial x_i} \right \|_{L^2(\Omega)} + \frac{C_{2M}}{a_0^3} \left\| \frac{\partial \theta}{\partial x_i}\right \|_{L^4(\Omega)} \left\| \frac{\partial \theta}{\partial x_j} \right\|_{L^4(\Omega)}.
    \end{equation*}
    This together with the fact that $\theta \in H^2(\Omega)\hookrightarrow W^{1,6}(\Omega) \hookrightarrow W^{1,4}(\Omega)$ yields that $y \in H^2(\Omega)$.
\end{proof}

\section*{Acknowledgments}
This work was supported by the DFG under the grants CL 487/2-1 and RO
2462/6-1, both within the priority programme SPP 1962 ``Non-smooth and
Complementarity-based Distributed Parameter Systems:
Simulation and Hierarchical Optimization''.

\printbibliography

\end{document}